\documentclass[a4paper,11pt]{amsart}
\usepackage{amssymb,amsmath,amsfonts,amscd,euscript}
\usepackage{hyperref}
\usepackage{color} 

\usepackage[usenames,dvipsnames]{xcolor} 
\usepackage{tikz, tikz-3dplot, pgfplots}
\usepackage{tikz-cd}

\numberwithin{equation}{section}
\newtheorem{trm}{Theorem}
\newtheorem{thm}{Theorem}[section]
\newtheorem{prop}[thm]{Proposition}
\newtheorem{lem}[thm]{Lemma}
\newtheorem{cor}[thm]{Corollary}
\newtheorem{rem}[thm]{Remark}

\newtheorem{example}[thm]{Example}
\newtheorem{dfn}[thm]{Definition}

\newcommand{\bC}{{\mathbb C}}
\newcommand{\bR}{{\mathbb R}}
\newcommand{\bP}{{\mathbb P}}

\newcommand{\bZ}{{\mathbb Z}}

\newcommand{\martina}[1]{\textcolor{red}{\textbf{#1}}}

\begin{document}

\title[Totally nonnegative Grassmannians and quiver Grassmannians]
{Totally nonnegative Grassmannians, Grassmann necklaces and quiver Grassmannians}

\author{Evgeny Feigin}
\address{E. Feigin:\newline
HSE University\\
Faculty of Mathematics\\
Usacheva 6\\Moscow 119048\\Russia\newline
{\it and }\newline
Skolkovo Institute of Science and Technology\\ 
Center for Advanced Studies\\
Bolshoy Boulevard 30, bld. 1\\
Moscow 121205\\
Russia
}
\email{evgfeig@gmail.com}
\author{Martina Lanini}
\address{M. Lanini:\newline Dipartimento di Matematica\\ Universit\`a di Roma ``Tor Vergata'',  Via della Ricerca Scientifica 1, I-00133 Rome, Italy}
\email{lanini@mat.uniroma2.it}
\author{Alexander P\"utz}
\address{A. P\"utz:\newline
Faculty of Mathematics\\
Ruhr-University Bochum\\
Universit\"atsstra\ss e 150\\
44780 Bochum\\
Germany}
\email{alexander.puetz@ruhr-uni-bochum.de}


\begin{abstract}
Postnikov constructed a cellular decomposition of the totally nonnegative Grassmannians. The poset of cells can be described (in particular) via  Grassmann necklaces.  
We study certain quiver Grassmannians for the cyclic quiver admitting a cellular decomposition, whose
cells  are naturally labeled by Grassmann necklaces. We show that the posets of cells 
coincide with the reversed cell posets of the cellular decomposition of the totally nonnegative Grassmannians. We investigate algebro-geometric and combinatorial properties of these quiver Grassmannians.
In particular, we describe the irreducible components,  study the action of the automorphism groups
of the underlying representations and describe the moment graphs. We also construct a resolution 
of singularities for each irreducible component; the resolutions are defined as  quiver
Grassmannians for an extended cyclic quiver.
\end{abstract}

\maketitle

\section*{Introduction}
Total positivity has a long story, starting in the first half of the 20th century \cite{GK37, Schoe47}. Thanks to Lusztig \cite{Lus94, Lus98a, Lus98b} it became of interest to Lie theorists and combinatorialists because of the relation with canonical bases. More precisely, Lusztig introduced the notion of totally non negative part of (real generalized) flag  varieties. The totally nonnegative (\emph{tnn} for short) Grassmannian ${\rm Gr}(k,n)_{\ge 0}$ is hence a special case of these and admits an easy description under the Pl\"ucker embedding. Namely,  ${\rm Gr}(k,n)_{\ge 0}$ is the subvariety of the (real) Grassmannian of $k$-dimensional subspaces of $\mathbb{R}^n$ whose points have Pl\"ucker coordinates all of the same sign (i.e. they can be normalized to be all nonnegative).


The tnn Grassmannians attracted a lot of attention due
to a large number of applications and various links with other mathematical structures, see \cite{Lam16} and 
references therein. In \cite{Pos06} Postnikov constructed a stratification of ${\rm Gr}(k,n)_{\ge 0}$ by the so called 
positroid cells, where each cell is isomorphic to a product of some copies of $\bR_{>0}$. There are many nice combinatorial
ways to label the positroid cells. In particular, the cells are in bijection with the so called Grassmann
necklaces: collections $(I_1,\dots,I_n)$ of $k$-element subsets of the set $\{1,\dots,n\}$ 
such that 
\[
I_1\ \subset I_2 \cup \{1\},\ I_2\ \subset I_3 \cup \{2\}, \dots,  I_n\subset I_1\cup\{n\}
\]
(a version of this definition leads to the notion of juggling pattern from \cite{KLS13}). 
 It is possible to equip the set of Grassmann necklaces with a combinatorial partial order, which turns out to agree with the reversed positroid cell closure relation (see e.g. \cite{W05}).
The same poset pops up in the study of the positroid varieties \cite{Lus98a,KLS13,KLS14}. These varieties stratify complex
Grassmannians and one can use the tools of algebraic geometry for their study. We note that the positroid varieties are no longer affine cells; Postnikov's positroid cells are obtained by intersecting positroid varieties with the tnn  Grassmannians. 

The notion of  Grassmann necklaces has a natural linear algebra analog. Namely, let us consider
an $n$-dimensional complex vector space $W$ with a fixed basis and the projections ${\rm pr}_i$ along the basis
vectors of $W$. Let $X(k,n)$ be the variety of collections of $k$-dimensional subspaces $V_1,\dots,V_n$ of
$W$ such that
\[
{\rm pr}_1V_1\subset V_2,\ {\rm pr}_2 V_2\subset V_3,\dots, {\rm pr}_n V_n\subset V_1 
\]
(this variety is denoted by ${\rm Jugg}$ in \cite{Kn08}). Our key observation is that $X(k,n)$ is isomorphic
to a quiver Grassmannian for the (equioriented) cyclic quiver $\Delta_n$ on $n$ vertices (Proposition~\ref{prop:Jugg-as-quiver-grass}).
More precisely, there exists a $\Delta_n$ module $U_{[n]}$ such that
\begin{equation}\label{eqn:IsoXknQuiverGrass}
X(k,n)\cong {\rm Gr}_{(k,\dots,k)}(U_{[n]}).
\end{equation}

In 1992, motivated by the study of quiver representations, Schofield introduced quiver Grassmannians \cite{Scho92}. Since then, they have been widely investigated, also in relation with positivity phenomena in cluster algebra theory, see the survey \cite{CI20} for an account of all progress achieved on quiver Grassmannian in these decades. It is shown in \cite{Re13} that every complex projective variety can be realized as a quiver Grassmannian. Finding a suitable quiver Grassmannian realization, as in our case, is hence desirable, as one can exploit all developed quiver techniques on top of the classical algebro-geometric tools. In particular, using \eqref{eqn:IsoXknQuiverGrass} we prove the following:

\begin{trm}\label{A}
$X(k,n)$ admits a cellular decomposition,  the cells are labeled by the Grassmann necklaces and the poset of cells is isomorphic to the reversed poset of the tnn Grassmannian cells.
\end{trm} 

In other words, the (complex) topology of $X(k,n)$ is very close to the (real) topology of the tnn Grassmannian. The realization \eqref{eqn:IsoXknQuiverGrass} allows to use the action of the automorphism group ${\rm Aut}_{\Delta_n}(U_{[n]})$ for the study of the topological properties of $X(k,n)$. We prove:

\begin{trm}
The cells from Theorem \ref{A} coincide with the  ${\rm Aut}_{\Delta_n}(U_{[n]})$ orbits. $X(k,n)$ is equipped with an $(n+1)$-dimensional torus action, which preserves the cellular decomposition. The corresponding moment graph has a combinatorial description in terms of Grassmann necklaces.  
\end{trm}

Finally, let us list the main algebro-geometric properties of our quiver Grassmannians. 

\begin{trm}
$X(k,n)$ has $\binom{n}{k}$ irreducible components, each of them of dimension $k(n-k)$. Each irreducible component of $X(k,n)$ is desingularized by an explicitely given quiver Grassmannian for an extended cyclic quiver $\widetilde{\Delta_n}$.
\end{trm}

\subsection*{Structure of the paper} In Section 1 and 2 we collect background material on tnn Grassmannians/Grassmann necklaces and quiver Grassmannians, respectively. Our main object $X(k,n)$ is introduced in Section 3, where we describe the key isomorphism \eqref{eqn:IsoXknQuiverGrass} and construct the cellular decomposition, as well as the irreducible components. In Section 4 we deal with several tori acting on $X(k,n)$, and in the following section we focus on one of them and give its moment graph. The Poincar\'e polynomial is determined in Section 6. Section 7 is about resolutions of singularities of the irreducible components. Finally, the two appendices are about linear degenerations and the $k=1$ case, respectively.

\subsection*{Acknowledgements} E.F. was partially supported by the grant RSF 19-11-00056. The study has been partially funded within the framework of the HSE University Basic Research Program. M.L. acknowledges the PRIN2017 CUP E8419000480006, as well as the MIUR Excellence Department Project awarded to the Department of Mathematics, University of Rome Tor Vergata, CUP
 E83C18000100006. 

\section{Totally nonnegative Grassmannians and positroids}\label{sec:tnn-Grass}
In this section we briefly recall basic facts on the totally nonnegative Grassmannians following \cite{Lam16}.  We discuss their geometry and combinatorics. In particular, we recall the definition of
the Grassmann necklaces and juggling patterns, which provide a bridge to the theory of quiver Grassmannians.

\subsection{Totally nonnegative Grassmannians}
The totally nonnegative (tnn for short) Grassmannian ${\rm Gr}(k,n)_{\ge 0}$ is the subset of the real Grassmann 
variety 
${\rm Gr}(k,n)$ represented by the subspaces whose Pl\"ucker coordinates have all the same sign. Postnikov \cite{Pos06} defined a cellular decomposition
of ${\rm Gr}(k,n)_{\ge 0}$ as follows. For $L\in {\rm Gr}(k,n)$ and a $k$-element subset $I\subset [n]=\{1,\dots,n\}$ 
let $X_I(L)$ be the $I$-th Pl\"ucker coordinate of $L$. We define 
\[
{\mathcal M}(L)=\left\{I\subset\binom{[n]}{k}:\ X_I(L)\ne 0\right\}.
\]
Then ${\mathcal M}(L)$ is a matroid attached to $L$. If $L\in {\rm Gr}(k,n)_{\ge 0}$, then the matroid ${\mathcal M}(L)$ is called 
a positroid. 

The importance of this notion is explained by the following theorem due to Postnikov \cite{Pos06}. 

\begin{thm}
Let ${\mathcal P}(k,n)$ be the set of positroids. For ${\mathcal M}\in {\mathcal P}(k,n)$ we denote by $\Pi({\mathcal M})\subset {\rm Gr}(k,n)_{\ge 0}$
the set of subspaces $L$ such that ${\mathcal M}(L)={\mathcal M}$. Then each stratum $\Pi({\mathcal M})$, ${\mathcal M}\in {\mathcal P}(k,n)$
is a cell $\bR_{> 0}^s$ ($s$ depends on the positroid ${\mathcal M}$).
\end{thm}

Hence one gets a cellular decomposition of the tnn Grassmannian
labeled by the positroids. It is thus natural to ask how to label the positroids ${\mathcal P}(k,n)$  and how to compute the dimension 
of the cell $\Pi({\mathcal M})$ for ${\mathcal M}\in {\mathcal P}(k,n)$. 

\begin{rem}
The cellular decomposition for totally nonnegative part of the flag varieties $G/P$ can be found in \cite{Lus94,Rie99,Rie06}. Postnikov's positroid decomposition
agrees with the general construction.
\end{rem}

\begin{rem}
It is important that we consider only the matroids corresponding to the points of the tnn Grassmannians. The strata corresponding to the
general matroids have much less transparent structure, see \cite{GGMS87} and the discussion in the introduction of \cite{KLS13}. 
\end{rem}

\subsection{Grassmann necklaces}\label{subsec:Grassmann-Necklace}
There are several ways to parametrize  the elements of ${\mathcal P}(k,n)$. The one providing the bridge between the theory of 
tnn Grassmannians and quiver Grassmannians for (equioriented) cyclic quivers is the following one (see \cite{Pos06}). 

\begin{dfn}\label{dfn:Grassmann-Necklace}
A $(k,n)$ Grassmann necklace is a collections $I_1,\dots,I_n$ of subsets of $[n]$ such that $|I_a|=k$ for all $a$ and 
$I_a\subset I_{a+1}\cup \{a\}$ for all $a=1,\dots,n$. The set of $(k,n)$ Grassmann necklaces is denoted by $\mathcal{GN}_{k,n}$.
\end{dfn}

We note that for $a=n$ the last condition is understood as $I_n\subset I_1\cup\{n\}$.
In other words, the condition on the sets $I_a$ can be written as $i\in I_a\setminus \{a\}$ implies
$i\in I_{a+1}$  (which works for $a=n$ as well). 
There is a slightly different version of Definition \ref{dfn:Grassmann-Necklace} 
(see e.g. \cite{KLS13}). Namely, given a collection
$(I_a)_{a\in [n]}\in \mathcal{GN}_{k,n}$ we define $J_a=\{i-a:\ i\in I_a\}$, where $i-a$ is
understood as an element of $[n]$, which is equal to $a-n$ modulo $n$. The resulting collections
$(J_a)_{a\in [n]}$ are called juggling patterns in \cite{KLS13} (modulo the overall change $j\to n+1-j$
of the elements of  $J_a$). Clearly, one can put forward the following definition. 

\begin{dfn}\label{dfn:Juggling-Pattern}
A collection $(J_a)_{a=1}^n$, $J_a\in \binom{[n]}{k}$ is called a juggling pattern if  
$j\in J_a\setminus\{n\}$ implies $j+1\in J_{a+1}$. 
\end{dfn}

\begin{rem}\label{rem:base-choice}
In Theorem~\ref{trm:cellular-decomposition} 
we show that Grassmann necklaces and juggling patterns naturally parametrize the 
torus fixed points in a certain quiver Grassmannian for the cyclic quiver. The two combinatorial
Definitions \ref{dfn:Grassmann-Necklace} and \ref{dfn:Juggling-Pattern} correspond to two natural
choices of basis in the representation space of the quiver (cf. 
Definition \ref{dfn:different-relisation-of-U_[n]}).
\end{rem}

The set of $(k,n)$ Grassmann necklaces can be equipped with a partial order. For two elements $I,J\in\binom{[n]}{k}$ such that $I=(i_1<\dots<i_k)$, $J=(j_1<\dots<j_k)$ 
we write $I\le J$ if $i_u\le j_u$ for all $u \in [k]$. Now for a number $a\in [n]$ we consider the rotated order 
\begin{equation}\label{a-order}
a<_a a+1 <_a \dots <_a n  <_a 1 <_a  \dots <_a a-1
\end{equation}
on the set $[n]$. This order induces the order $<_a$ on the set $\binom{[n]}{k}$.
Now for two $(k,n)$ Grassmann necklaces ${\mathcal I}=(I_1,\dots,I_n)$ and ${\mathcal J}=(J_1,\dots,J_n)$ we write 
${\mathcal I}\le {\mathcal J}$ if $I_a\le_a J_a$ for all $a\in [n]$.
\begin{example}\label{exple:PosetGN13}
Let $k=1$ and $n=3$. Given a Grassmann necklace $(I_1=\{i_1\},I_2=\{i_2\},I_3=\{i_3\})\in\mathcal{GN}_{1,3}$, we represent such an element by $i_1i_2i_3$. In this case the Hasse diagram of the poset $(\mathcal{GN}_{1,3},\leq)$ is the following:

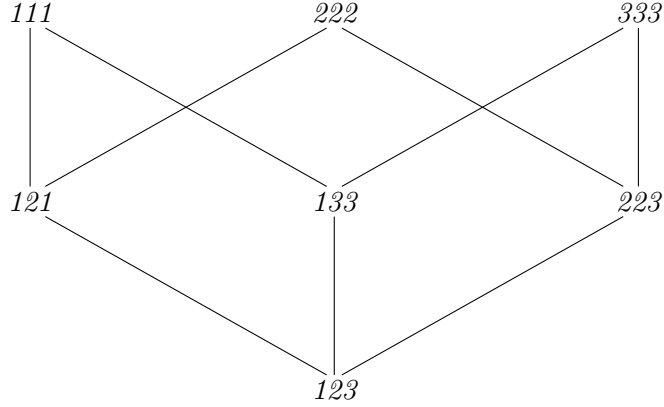
\begin{center}
\begin{tikzpicture}
\node at (0,0) {123};
\node at (0,2.5) {133};
\node at (0,5) {222};
\node at (-4,2.5) {121};
\node at (-4,5) {111};
\node at (4,2.5) {223};
\node at (4,5) {333};

\draw(-3.8,2.3) -- (-.1,0.2);
\draw (3.8,2.3) -- (.1,0.2);
\draw (0,2.3) -- (0,0.2);

\draw (-4,4.8) -- (-4,2.7);
\draw (-3.8,4.8) -- (-0.1,2.7);

\draw(4,4.8) -- (4,2.7);
\draw(3.8,4.8) -- (0.1,2.7);

\draw(0.1,4.8) -- (3.8,2.7);
\draw(-0.1,4.8) -- (-3.8,2.7);
\end{tikzpicture}
\end{center}
\end{example}

Given a positroid ${\mathcal M}\in {\mathcal P}(k,n)$ we define the corresponding Grassmann necklace ${\mathcal I}({\mathcal M})$ by the formula
\[
{\mathcal I}({\mathcal M})_a=\min{_a} \{J\in{\mathcal M}\}, 
\]
where $\min_a$ is the minimum with respect to the order $\le _a$. 

\begin{prop} \cite[Theorem 7.12]{Lam16}
The map ${\mathcal M}\mapsto {\mathcal I}({\mathcal M})$ is an order reversing bijection between the set of $(k,n)$ positroids and the set
of $(k,n)$ Grassmann necklaces. 
\end{prop}

\subsection{Bounded affine permutations}\label{subsec:BoundedPerms}
We recall here briefly the definition of bounded affine permutations and their relation with Grassmann necklaces. More details can be found in \cite[\S3]{KLS13} and \cite[\S6]{Lam16}.

Recall that each stratum $\Pi({\mathcal M})$ is a cell $\bR_{> 0}^d$. In order to give a formula for  the dimension $d$ of the cell,
we use one more parametrization via the bounded affine permutations. A $(k,n)$ affine permutation (not yet bounded) is a bijection 
$f:\bZ\to\bZ$ satisfying the following properties:
\begin{itemize}
    \item $f(i+n)=f(i)+n$ for all $i\in\bZ$,
    \item $\sum_{i=1}^n (f(i)-i)=kn$.
    \end{itemize}
In particular, there is a distinguished $(k,n)$ affine permutation ${\rm id}_k$ given by ${\rm id}_k(i)=i+k$. 
The length of an affine permutation is defined as 
\[
l(f)=|\{(i,j)\in[n]\times\bZ:\ i<j \text{ and } f(i)>f(j)\}|.
\]
We note that the set of $(0,n)$ affine permutations is a group isomorphic to the affine Weyl group $W_n$ of type $A_{n-1}^{(1)}$.
For general $k=1,\dots,n$ the group $W_n$ acts freely and transitively on the set of $(k,n)$ affine permutations; the action of the permutation 
$s_i=(i,i+1)\in W_n$, for $i=0,\dots,n-1$ permutes the values $f(i+rn)$ and $f(i+rn+1)$ for all $r\in\bZ$. This allows to identify
the set of $(k,n)$ affine permutations with $W_n$ by sending $w\in W_n$ to $w.{\rm id}_k$. We thus obtain an order $\le$
on the set of $(k,n)$ affine permutations coming from the Bruhat order on $W_n$. For example, the unique minimal element is 
${\rm id}_k$. 

A $(k,n)$ bounded affine permutation is a $(k,n)$ affine permutation subject to the extra condition
\[
i\le f(i)\le i+n \text{ for all }  i\in\bZ.
\]
We denote the set of $(k,n)$ bounded affine permutations by ${\mathcal B}_{k,n}$.
\begin{example}
The Grassmann necklace corresponding to ${\rm id}_k$ is the one defined by
\[
I_a=(a, a+1, \ldots, a+k-1) \qquad (a\in[n]).
\]
\end{example}

It is shown in \cite{KLS13} that ${\mathcal B}_{k,n}$
is a lower order ideal in the set of  $(k,n)$  affine permutations (unbounded). For $f,g\in {\mathcal B}_{k,n}$ we write $f\le g$ for 
the induced order.
The Grassmann necklace ${\mathcal I}(f)=(I_1,\dots,I_n)$ for $f\in {\mathcal B}_{k,n}$ is defined by the formula
\[
I_a=\{f(b):\ b<a \text{ and } f(b)\ge a\} \mod n.
\]
By \cite[Theorem 6.2]{Lam16}, this defines an order preserving bijection between the set ${\mathcal B}_{k,n}$ and the set of $(k,n)$ Grassmann necklaces. 

In the opposite direction, given a Grassmann necklace ${\mathcal I}$ we define the bounded affine permutation
$f=f({\mathcal I})$ as follows: if $a\notin I_a$, then $f(a)=a$. If $a\in I_a$, then 
$I_{a+1}=I_a\setminus\{a\}\cup\{b\}$. We define $f(a)=c$, where $b\equiv c \mod n$ and $a<c\le a+n$.   
\begin{rem}\label{rem:cell-dim-tnn-Grass}
Let ${\mathcal M}(f)$ be the positroid defined by $f\in {\mathcal B}_{k,n}$, that is the one obtained from $\mathcal{I}(f)$ as explained in \S\ref{subsec:Grassmann-Necklace}. Then (see e.g. \cite{Lam16})
\begin{itemize}
    \item $\dim \Pi_{{\mathcal M}(f)}=k(n-k)-l(f)$,
    \item the closure of the cell $\Pi_{{\mathcal M}(f)}$ contains $\Pi_{{\mathcal M}(g)}$ if and only if $f\ge g$.    
\end{itemize}
\end{rem}

\begin{rem}
There is a complex version of the cellular decomposition of the tnn Grassmannians. Namely, one can define 
the stratification of the complex Grassmann varieties into the positroid varieties \cite{KLS13}. The latter ones 
 are not isomorphic to affine cells in general, but they are irreducible complex projective  algebraic varieties with many nice properties.  
\end{rem}

\section{Quiver Grassmannians for cyclic quivers}\label{sec:quiver-Grass}
In this short section we recall some definitions and results concerning quiver Grassmannians and discuss the equiorented cycle case. Later, we will relate certain quiver Grassmannians for the cycle to totally non-negative Grassmannians.
\subsection{Quivers and Representations}
A finite quiver $Q$ consists of a finite set of vertices $Q_0$ and a finite set of arrows $Q_1$. Each $\alpha \in Q_1$ has a unique source and target $i,j \in Q_0$ and we write $(\alpha: i \to j)$. A finite dimensional $Q$ representation $M$ is a pair of tuples $(M^{(i)})_{i \in Q_0}$ and $(M_\alpha)_{\alpha \in Q_1}$, where each $M^{(i)}$ is a finite dimensional $\mathbb{C}$-vector space and each $M_\alpha$ is a linear map from $M^{(i)}$ to $M^{(j)}$. The notion of a subrepresentation will be fundamental to us: a tuple of finite dimensional $\bC$-vector spaces $N=(N^{(i)})_{i \in Q_0}$ is a subrepresentation of $M=((M^{(i)})_{i\in Q_0},(M_\alpha)_{\alpha\in Q_1})$ if $N^{(i)}\subset M^{(i)}$ for all $i\in Q_0$ and  $M_\alpha N^{(i)} \subset N^{(j)}$ for all $\alpha:i\to j \in Q_1$.

On the other hand, it might be useful to define a subrepresentation as a subobject in the  appropriate quiver representation category. For this we need the notion of $Q$ morphism: a morphism from the $Q$ representation $M$ to the $Q$ representation $N$ is a tuple of linear maps $\varphi = (\varphi_i)_{i\in Q_0} \in \prod_{i \in Q_0} \mathrm{Hom}_\mathbb{C}(M^{(i)},N^{(i)})$ such that: 

\begin{center}
\begin{tikzpicture}

	\draw[arrows={-angle 90}, shorten >=9, shorten <=9]  (-.1,0) -- (-.1,-1.6);
	\draw[arrows={-angle 90}, shorten >=9, shorten <=13]  (0,0) -- (1.6,0);
	\draw[arrows={-angle 90}, shorten >=9, shorten <=9]  (1.6,0) -- (1.6,-1.6);
	\draw[arrows={-angle 90}, shorten >=9, shorten <=13]  (0,-1.6) -- (1.6,-1.6);
	
	\node at (0,0) {$M^{(i)}$};
	\node at (0,-1.6) {$M^{(j)}$};
    \node at (1.7,0) {$N^{(i)}$};
   \node at (1.7,-1.6) {$N^{(j)}$};
    
    \node at (-.42,-.8) {$M_{\alpha}$};
	\node at (0.8,.3) {$\varphi_{i}$};
    \node at (2.0,-.8) {$N_{\alpha}$};
    \node at (0.8,-1.9) {$\varphi_{j}$};
    
    \node[rotate=45] at (.8,-.8) {$\equiv$};
	 
\end{tikzpicture}
\end{center}

By $\mathrm{Hom}_Q(M,N)$ we denote the set of all $Q$ morphisms from $M$ to $N$. The category of finite-dimensional $Q$ representations over $\mathbb{C}$ is denoted by $\mathrm{rep}_\mathbb{C}(Q)$. Now $N \in \mathrm{rep}_\mathbb{C}(Q)$ is a subrepresentation of $M \in \mathrm{rep}_\mathbb{C}(Q)$ if $\mathrm{Hom}_Q(N,M)$ contains a $Q$ monomorphism, i.e. $N$ is a subobject of $M$. 
The dimension vector of $M \in \mathrm{rep}_\mathbb{C}(Q)$ is 
\[ \mathbf{d} := \big( \dim_\mathbb{C} M^{(i)} \big)_{i \in Q_0} \in \mathbb{Z}_{\geq 0}^{Q_0}. \]
\begin{dfn}
For $M \in \mathrm{rep}_\mathbb{C}(Q)$ and $\mathbf{e} \in \mathbb{Z}^{Q_0}$, the quiver Grassmannian $\mathrm{Gr}_\mathbf{e}(M)$ is the variety of all subrepresentations of $M$ whose dimension vector equals $\mathbf{e}$.
\end{dfn}
\subsection{The path algebra of a quiver}\label{sec:path-algebra}
A path $p$ in a quiver $Q$ is a concatenation of consecutive arrows. We define the source of a path as the source of its first arrow and its target is the target of the last arrow. The path algebra $\mathbb{C}Q$ has all paths in $Q$ as basis and the multiplication $*$ of two paths $p$ and $p'$ is defined via concatenation: If the target of $p$ is the source of $p'$, then $p'*p:= p' \circ p$. Otherwise the product is zero.

Using paths in $Q$ we can define a set of relations $R$ on the objects of $\mathrm{rep}_\mathbb{C}(Q)$. Let $I \subset \mathbb{C}Q$ be the ideal generated by the relations in $R$. Then there is an equivalence of categories between $\mathrm{rep}_\mathbb{C}(Q,I)$, i.e. representations satisfying the relatations in $I$ (called bounded quiver representations) and $\mathrm{mod}_\mathbb{C}(\mathbb{C}Q/I)$, i.e. modules over the so called bounded path algebra \cite[Theorem~5.4]{Schiffler2014}. 

The injective bounded representation $I_k$ at the vertex $k \in Q_0$ consists of the vector spaces $V^{(j)}$ for $j \in Q_0$ with a basis indexed by equivalence classes of paths (in $\mathbb{C}Q/I$) from $j$ to $k$. The linear map $V_\alpha$ along the arrow $(\alpha : i \to j) \in Q_1$ sends a basis element of $V^{(i)}$ indexed by the equivalence class of a path $p$ to a basis element of $V^{(j)}$ which is indexed by an other equivalence class with representative $p'$ such that $p = \alpha \circ p'$. A basis element of $V^{(i)}$ is send to zero by the map along the arrow $(\alpha : i \to j)$ if the paths in the equivalence class indexing this basis element do not factor through $\alpha$.

\subsection{The equioriented Cycle}\label{sec:equi-cycle}
Let $\Delta_n$ denote the equioriented cycle on $n$ vertices, where the orientation is chosen in such a way that $1\to 2$ is an arrow. Then the set of vertices and the set of arrows are both in bijection with $\mathbb{Z}_n := \mathbb{Z}/n \mathbb{Z}$. Unless specified differently, we consider all indices of vertices and arrows modulo $n$. Given a representation $M\in\mathrm{rep}_\mathbb{C}(\Delta_n)$, we 
write $M_{\alpha_i}$ for $M_{i\to i+1}$ for any $i\in\bZ_n$.

Every point $M$ of the affine variety
\[ {\rm R}_\mathbf{n}(\Delta_n):= \bigoplus_{i \in \mathbb{Z}_n} \mathrm{Hom}_\mathbb{C}(\mathbb{C}^n,\mathbb{C}^n)\]
parametrizes a $\Delta_n$ representation and each $g \in {\rm G}_\mathbf{n} := \prod_{i \in \mathbb{Z}_n} GL_n(\mathbb{C})$ acts on ${\rm R}_\mathbf{n}(\Delta_n)$ via conjugation
\[ g.M := \big( g_{i+1} M_{\alpha_i} g_i^{-1} \big)_{i \in \bZ_n}.\]
 The automorphism group $\mathrm{Aut}_{\Delta_n}(M)$ of a $\Delta_n$ representation $M \in {\rm R}_\mathbf{n}(\Delta_n)$ is its stabilizer in ${\rm G}_\mathbf{n}$.

Let $M$ be a $Q$ representation. A basis of $M$ is  a basis $B$ of the underlying vector space $\bigoplus_{i\in\ Q_0}M^{(i)}$. In the case of $Q=\Delta_n$, we will always pick bases $B$ of $M$ compatible with the $\bZ_n$-grading on $\bigoplus_{i\in \bZ_n}M^{(i)}$: $B=\bigcup_{i\in\bZ_n} B^{(i)}$, where $B^{(i)}=\{w_1^{(i)}, \ldots w_{d_i}^{(i)}\}$ is a basis for $M^{(i)}$.

\begin{dfn}\label{dfn:coefficient-quiver}
Let $M \in \mathrm{rep}_\mathbb{C}(\Delta_n)$ and $B$ be a basis of $M$. The coefficient quiver $Q(M,B)$ consists of: 
\begin{itemize}
\item[(QM0)] the vertex set  $Q(M,B)_0=B$,
\item[(QM1)] the set of arrows $Q(M,B)_1$, containing $\big(\alpha: w_k^{(i)} \to w_\ell^{(i+1)}\big)$ if and only if the coefficient of $w_\ell^{(i+1)}$ in $M_{\alpha_i} w_k^{(i)}$ is non-zero.
\end{itemize}
\end{dfn}
\begin{example}\label{exple:UinCoeffQuiver}
Let $U(i;n)$ be the $\Delta_n$ representation given by 
\[
U(i;n)^{(j)}=\bC, \qquad U(i;n)_{\alpha_j}=\left\{
\begin{array}{ll}
 \mathrm{id}_{\bC}    &  \hbox{ if }j\neq i,\\
0     & j=i.
\end{array}
\right.
\]
 If we denote by $w^{(j)}$ a generator of $U(i;n)^{(j)}$ (that is, any non zero element), then the corresponding coefficient quiver is just an equioriented type $A_n$ Dynkin quiver:
\[
w^{(i+1)}\to w^{(i+2)}\to\ldots\to w^{(i-1)}\to w^{(i)}.
\]
\end{example}
Given a fixed basis $B$ of $M\in \mathrm{rep}_\mathbb{C}(\Delta_n)$, a grading of $M$ is simply a map $\mathrm{wt} : B \to \mathbb{Z}^B$. This induces a $\mathbb{C}^*$ action on the vector spaces of $M$, defined on the basis $B$ as follows and then extended by linearity: 
\begin{equation}\label{rem:C*-action}
    \lambda.b := \lambda^{\mathrm{wt}(b)}\cdot b\qquad (b\in B, \lambda\in\bC^*).
\end{equation} 

With some additional assumptions about the grading (see \cite[Section~5.1]{LaPu2020}), the $\mathbb{C}^*$ action extends to the quiver Grassmannian $\mathrm{Gr}_\mathbf{e}(M)$ with finitely many fixed points 
\[ \{ L_1, \dots, L_m \} =: \mathrm{Gr}_\mathbf{e}(M)^{\mathbb{C}^*}\]
indexed by appropriate subquivers of $Q(M,B)$ (see \cite[Proposition 1]{Cerulli2011}). Moreover, the $\bC^*$ action   
 induces an $\alpha$-partition of $\mathrm{Gr}_\mathbf{e}(M)$ into the attracting sets of the fixed
 points
\[
W_L := \big\{ U \in \mathrm{Gr}_\mathbf{e}(M) : \lim_{\lambda \to  0} \lambda.U = L  \big\}, \]
i.e. there exists a total order on the fixed point set such that $\bigsqcup_{i=1}^s W_{L_i} $ is closed in $\mathrm{Gr}_\mathbf{e}(M)$ for all $s \in [m]$.
To prove the existence of a cellular decomposition of the quiver Grassmannian it remains to show that all the $W_L$'s are isomorphic to affine spaces.

\section{The main object}\label{sec:main-object}
In this section, we identify $X(k,n)$ from the introduction with a certain quiver Grassmannian and apply methods from representation theory of quivers to investigate its geometric properties.

For convenience, we start by recalling the definition of $X(k,n)$. Let $W$ be an $n$-dimensional $\bC$-vector space and let $(e_1, \ldots, e_n)$ be a basis, then 
\[
X(k,n)=
\left\{
(V_i)\in \prod_{i\in \bZ_n}{\rm Gr}_k(W)\mid {\rm pr}_i V_i\subseteq V_{i+1}
\right\},\]
where ${\rm Gr}_k(W)$ denotes the usual Grassmann variety of $k$-dimensional subspaces of $W$, and the projection  morphisms are defined as ${\rm pr}_i(e_j)=e_j$ for any $j\neq i$ and ${\rm pr}_i(e_i)=0$ for all $i\in\bZ_n$.

Moreover, recall $U(i;n)$ from Example \ref{exple:UinCoeffQuiver}, i.e. the $\Delta_n$ representation which is one-dimensional over each vertex $i \in \mathbb{Z}_n$ and the map along the arrow $j \to j+1$ is the identity for $j \neq i$ and zero for $i \to i+1$. 

The following result tells us that the variety $X(k,n)$ can be realized as a quiver Grassmannian for a very special $\Delta_n$ representation. Namely, let  $$U_{[n]}=\bigoplus_{i \in \mathbb{Z}_n} U(i;n).$$

\begin{prop}\label{prop:Jugg-as-quiver-grass}
Let $k,n \in \mathbb{N}$ with $k < n$
and let ${\bf k}=(k,\dots,k) \in \mathbb{Z}^n$.
Then
\[
X(k,n) \cong {\rm Gr}_{\bf k}(U_{[n]}).
\]
\end{prop} 

\begin{proof}
By definition, $U(i;n)$ is isomorphic to the representation $V$ with $V^{(j)} = \mathbb{C}$ for all $j \in \mathbb{Z}_n$ and $V_{\alpha_j} = {\rm id}_\mathbb{C}$ for all $j \in \mathbb{Z}_n$ with $j \neq i$ and $V_{\alpha_i} = 0$. Hence the vector spaces of $U_{[n]}$ over the vertices of $\Delta_n$ are all $n$ dimensional and with a suitable order of the direct summands of $U_{[n]}$ we obtain
\[ U_{[n]} \cong  M := \Big( \big( M^{(i)} =\mathbb{C}^n\big)_{i \in \mathbb{Z}_n}, \big( M_{\alpha_i} = {\rm pr}_i \big)_{i \in \mathbb{Z}_n} \Big).\]
where ${\rm pr}_i$ sends the $i$-th basis vector of $\mathbb{C}^n$ to zero and preserves the remaining. 
This implies ${\rm Gr}_{\bf k}(U_{[n]}) \cong {\rm Gr}_{\bf k}(M)$ and the desired isomorphism follows from the definition of $X(k,n)$.
\end{proof}

\begin{rem}\label{rem:bound-quiver-relations}
$U_{[n]}$ is a representation for the bounded quiver $\Delta_n$ with the relation that all length $n$ loops vanish. Let $I$ be the ideal of the path algebra $\mathbb{C} \Delta_n$ generated by all paths of length $n$, then we can view $U_{[n]}$ as a module over the bounded path algebra $\mathbb{C} \Delta_n/I$ (cf. \cite[§~2.2, Proposition~4.1]{Pue2020}). 
\end{rem}

\begin{dfn}\label{dfn:nilpot-rep}
We call $M \in {\rm rep}_\mathbb{C}(\Delta_n)$ nilpotent if all concatenations of the maps of $M$ along cyclic paths vanish beyond a certain length of the paths.
\end{dfn}
\subsection{The Automorphism Group of $U_{[n]}$} The explicit realization of the group ${\rm Aut}_{\Delta_n}(U_{[n]})$ as  subgroup of ${\rm G}_\mathbf{n}$ depends on the basis of $U_{[n]}$. 

There are two special bases which we use throughout this paper:
\begin{dfn}\label{dfn:different-relisation-of-U_[n]}
\begin{enumerate}
    \item The first basis is compatible with the choice made in the proof of Proposition \ref{prop:Jugg-as-quiver-grass}: for any $i\in\bZ_n$ we set
    \[
    B^{(i)}=\big\{b_1^{(i)}, \ldots, b_n^{(i)}\big\}
    \]
    so that $(U_{[n]})_{\alpha_i}(b_j^{(i)})=\left\{
    \begin{array}{ll}
    b_j^{(i+1)}     &\hbox{ if }j\neq i,  \\
    0     & \hbox{ if }j=i.
    \end{array}\right.$ 
    
    We will borrow notation from Proposition \ref{prop:Jugg-as-quiver-grass} and write ${\rm pr}_i$ for $(U_{[n]})_{\alpha_i}$ with respect to the above basis.
    This basis will allow us to relate $X(k,n)$ to Grassmann necklaces.
    \item By rearranging the previous basis vectors\footnote{More precisely, we reorder any set $\{b_1^{(i)}, \ldots b_n^{(i)}\}$  decreasingly with respect to the shifted total order $\leq_i$ \eqref{a-order}}, we get
    \[
    B^{(i)}=\{v_1^{(i)}, \ldots, v_n^{(i)}\}
    \]
    and with respect to this basis
    we have  \[(U_{[n]})_{\alpha_i}(v_j^{(i)})=\left\{
    \begin{array}{ll}
    v_{j+1}^{(i+1)}     &\hbox{ if }j\neq n,  \\
    0     & \hbox{ if }j=n.
    \end{array}\right. .\] We denote this morphism by $s_1$.
    This basis will allow us to relate $X(k,n)$ to juggling patterns. From now on we will work with this basis most of the time.
\end{enumerate}
  \end{dfn}

Observe that the choice of a basis corresponds to a certain realization of $U_{[n]}$ as a point in ${\rm R}_\mathbf{n}(\Delta_n)$. 

\begin{rem}\label{secondrealization}
The second realization of  $U_{[n]}$ from the definition above leads to the following  realization of $X(k,n)$ (juggling patterns style). Let $W'$ be an $n$-dimensional $\bC$-vector space and let $(v_1, \ldots, v_n)$ be a basis of $W'$, then 
\[
X(k,n)=
\left\{
(V_i)\in \prod_{i\in \bZ_n}{\rm Gr}_k(W')\mid s_1(V_i)\subseteq V_{i+1}
\right\},\]
where $s_1(v_j)=v_{j+1}$ for any $j\neq n$ and $s_1(v_n)=0$.
\end{rem}

If $M \in {\rm R}_\mathbf{n}(\Delta_n)$, then its endomorphism algebra ${\rm End}_{\Delta_n}(M)$ is defined as the set of matrix tuples $ E = (E_i)_{i \in \mathbb{Z}_n} \prod_{i \in \mathbb{Z}_n} {\rm M}_n(\mathbb{C})$ such that
\[
E_{i+1}M_i = M_i E_i \quad \mathrm{for} \ \mathrm{all} \ i \in \mathbb{Z}_n. 
\]

\begin{prop}\label{prop:endomorphism-algebra}
With respect to the basis $\bigcup_{i\in\bZ_n}\{v_1^{(i)}, \ldots, v_n^{(i)}\}$, the elements of the endomorphism algebra ${\rm End}_{\Delta_n}(U_{[n]})$
 are exactly the matrix tuples $ E = (E_i)_{i \in \mathbb{Z}_n}$ with
\[ 
E_i = \begin{pmatrix}
e^{(i)}_{1,1} & & & &\\
e^{(i)}_{2,1}& e^{(i-1)}_{1,1} & & &\\
\vdots & \vdots & \ddots& &\\
e^{(i)}_{n-1,1} & e^{(i-1)}_{n-2,1} & \hdots & e^{(i-n+2)}_{1,1} &\\
e^{(i)}_{n,1} & e^{(i-1)}_{n-1,1} & \hdots & e^{(i-n-2)}_{2,1} & e^{(i-n+1)}_{1,1}
\end{pmatrix} 
\]
where $e^{(i)}_{k,1} \in \mathbb{C}$ for all $i \in \mathbb{Z}_n$, $k \in [n]$. In particular, $\dim_\bC {\rm End}_{\Delta_n}(U_{[n]})=n^2$.
\end{prop} 

\begin{proof}
By definition of ${\rm End}_{\Delta_n}(U_{[n]})$, we have that $(E_i)_{i\in\bZ_n}$ if and only if
\[ E_{i+1} s_1 = s_1 E_i \quad \mathrm{for} \ \mathrm{all} \ i \in \mathbb{Z}_n, \]
that is 
\begin{equation}\label{eqn:AutGp}
    E_{i+1} s_1(v_l^{(i)}) = s_1 E_i(v_l^{(i)}) \quad \mathrm{for} \ \mathrm{all} \ i,l \in \mathbb{Z}_n.
\end{equation}  
Let us we write $e^{(i)}_{k,l}:= (E_i)_{k,l}$, so that $E_i(v_l^{(i)})=\sum_{k=1}^n e_{k,l}^{(i)}v_k^{(i)}$.
It is then easy to see that the equations \eqref{eqn:AutGp} are equivalent to 
\[
e^{(i)}_{k,l} = e^{(i+1)}_{k+1,l+1},\quad e^{(i)}_{k,n}= 0, \quad e^{(i)}_{n,l}= 0, \qquad \mathrm{for} \ \mathrm{all } \ k,l \in [n-1].
\]
From the previous equations, it follows by induction on $n-l$ that $e_{k,l}^{(i)}=0$ for any $l>k$, and by induction on $l$ that  $e_{k,l}^{(i)}=e_{k+1,l+1}^{(i+1)}$.
 This implies that the $E_i$'s are of the desired form.
\end{proof}

\begin{rem}
We obtain ${\rm Aut}_{\Delta_n}(U_{[n]}) \subset {\rm End}_{\Delta_n}(U_{[n]})$ by the condition $e^{(i)}_{1,1} \neq 0$ for all $i \in \mathbb{Z}_n$, because $E_{i+1} s_1 = s_1 E_i$ is equivalent to $s_1 = E_{i+1} s_1 E_i^{-1}$ if the matrices $E_i$ are invertible. 
\end{rem}

\begin{rem}We will see in \S\ref{sec:T-action} that the automorphism group of $X(k,n)$ is larger than the automorphism group of $U_{[n]}$. 
\end{rem}

\subsection{Geometric Properties of the Main Object}We prove here, by applying quiver representation theory results,  geometric properties of $X(k,n)$.
\begin{thm}\label{trm:irred-comp-and-dim} Let  $k,n \in \mathbb{N}$ with $k < n$ then:
\begin{enumerate}
\item $X(k,n)$ has $\binom{n}{k}$ irreducible components;
\item $X(k,n)$ is equidimensional of dimension $k(n-k)$.
\end{enumerate}
\end{thm}

\begin{proof}
By \cite[Lemma~4.10]{Pue2020}, the irreducible components of $X(k,n)$ are parametrized by the set
\[ \left\{ p = (p_i)_{i \in \mathbb{Z}_n} \in \prod_{i \in \mathbb{Z}_n} \{ 0,1 \} \ : \ \sum_{i \in \mathbb{Z}_n} p_i = k \right\} \]
and they are all of dimension $k(n-k)$. The above set is in bijection with $\binom{[n]}{k}$, i.e. the set containing all $k$-element subsets of $[n] := \{1,\dots,n\}$.
\end{proof}
\begin{example}
Let $k=1$ and $n=3$. Observe that in this case 
\[
X(1,3)\simeq X:=\left\{
\left(
\begin{bmatrix}
a_i\\b_i\\c_i
\end{bmatrix}
\right)\in\prod_{i\in\bZ_3}\mathbb{P}^2\mid {\rm pr}_i\left(\begin{matrix}
a_i\\b_i\\c_i
\end{matrix}\right)\in\bC\left(\begin{matrix}
a_{i+1}\\b_{i+1}\\c_{i+1}
\end{matrix}\right), \ i\in\bZ_3
\right\}
\]
The three irreducible components are
\[
X_i=\left\{
\left(
\begin{bmatrix}
a_i\\b_i\\c_i
\end{bmatrix}
\right)\in X\mid \begin{bmatrix}
a_i\\b_i\\c_i
\end{bmatrix}=[b^{(i)}_i]
\right\}\qquad (i\in\bZ_3)
\]
where $[b^{(i)}_i]$, according to Definition \ref{dfn:different-relisation-of-U_[n]}(1), denotes the class of the $i$-th standard basis vector of $\bC^3$. For example,
\begin{align*}
    X_1&=
\left\{
    \left(
\begin{bmatrix}
1\\0\\0
\end{bmatrix},\begin{bmatrix}
a_2\\b_2\\c_2
\end{bmatrix},\begin{bmatrix}
a_3\\0\\c_3
\end{bmatrix}
\right)\in\prod_{i\in\bZ_3}\mathbb{P}^2\mid a_2c_3-a_3c_2=0 
\right\}.
\end{align*}
We see immediately that $X_1$ is a projective variety of dimension 2.
\end{example}

\begin{thm}\label{trm:cellular-decomposition}For $k,n \in \mathbb{N}$ with $k < n$ the following holds:  
\begin{enumerate}
\item $X(k,n)$ admits a cellular decomposition;
\item The cells are naturally labeled by the $(k,n)$ Grassmann necklaces.
\end{enumerate}
\end{thm}
\begin{proof}
Let $B = \cup_{i \in \mathbb{Z}_n} B^{(i)}$ be the first basis of $U_{[n]}$  
 from Definition \ref{dfn:different-relisation-of-U_[n]}. 
 We fix the weight function ${\rm wt}(b^{(i)}_j):= j$ for all $i \in \mathbb{Z}_n$ and all $j \in [n]$. Hence by \cite[Proposition~1]{Cerulli2011}, the $\mathbb{C}^*$ fixed points are parametrized by the elements of 
\[ \left\{ I = (I_j)_{j \in \mathbb{Z}_n} \in  \prod_{i \in \mathbb{Z}_n} \binom{[n]}{k} \ : \ I_j \setminus\{j\} \subset I_{j+1} \ \mathrm{for} \ \mathrm{all} \ j \in \mathbb{Z}_n \right\}. \]
This set coincides 
 with the set $\mathcal{GN}_{k,n}$ of Grassmannian necklaces.

With this $\mathbb{C}^*$ action not all attracting sets of the fixed points are isomorphic to affine spaces. For this reason we switch to the second basis 
of $U_{[n]}$ from 
Definition \ref{dfn:different-relisation-of-U_[n]}
but keep the same weight function. Observe that now the weight difference along each arrow of the coefficient quiver is one whereas it was zero for the first choice of a basis of $U_{[n]}$.
Now by 
\cite[Proposition~1]{Cerulli2011}, the fixed points of the induced $\mathbb{C}^*$ action are exactly the juggling patterns as in Definition~\ref{dfn:Juggling-Pattern}.
Using the second  
 basis of $U_{[n]}$ and the $\mathbb{C}^*$ action described above, the first part is a special case of \cite[Theorem~4.13]{Pue2020}.
\end{proof}

\begin{rem}
The two different bases from Definition \ref{dfn:different-relisation-of-U_[n]} lead to the parametrization of the cells via 
 $(k,n)$ Grassmann necklaces and juggling patterns, respectively.
\end{rem}


For a point $U \in {\rm Gr}_{\bf k}(M)$ the isomorphism class $\mathcal{S}_U$ in the quiver Grassmannian is called stratum and is irreducible by \cite[Lemma~2.4]{CFR12}.

\begin{rem}\label{strata}
The closures of the top-dimensional strata are the irreducible components of $X(k,n)={\rm Gr}_{\bf k}(U_{[n]})$ and by \cite[Lemma~4.10]{Pue2020} these strata have the representatives 
\[ U_J = \bigoplus_{j \in J} U(j;n) \quad \mathrm{for} \ J \in \binom{[n]}{k}. \]
\end{rem}

\begin{prop}\label{prop:irred-comp-are-quiver-Grass}
For all $J \in \binom{[n]}{k}$, the irreducible component $X_J(k,n):=\overline{\mathcal{S}_{U_J}}$ of $X(k,n)$ is isomorphic to the quiver Grassmannian ${\rm Gr}_{\bf k}(M_J)$ where $M_J := {\rm End}_{\Delta_n}(U_{[n]}).U_J$. 
\end{prop} 
\begin{proof}
We consider here the second basis $\bigcup_{i\in\bZ_n}\{v_1^{(i)}, \ldots, v_n^{(i)}\}$ of $U_{[n]}$ from Definition \ref{dfn:different-relisation-of-U_[n]}. The coefficient quiver of each $M_J$ with respect to this basis is hence the full subquiver of $Q(U_{[n]},B)$ on the vertices with indices greater or equal to the indices of $U_J$ viewed as subsegments in $Q(U_{[n]},B)$. This is a direct consequence of the structure of ${\rm End}_{\Delta_n}(U_{[n]})$ as described in Proposition~\ref{prop:endomorphism-algebra}. Each ${\rm Gr}_{\bf k}(M_J)$ admits a cellular decomposition by \cite[Theorem~4.13]{Pue2020}. Now we compute that all cells in $\overline{\mathcal{S}_{U_J}}$ are described by the same equations as the corresponding cells of ${\rm Gr}_{\bf k}(M_J)$. Hence with this explicit coordinate description the desired isomorphism is just the identity map. 
\end{proof} 

\begin{rem}\label{rem:non-smoothness-irreducible components}
In general, the quiver Grassmannian $X_J(k,n)$ is not smooth. For example consider the case $n=5$, $k=3$ and $J = \{2,4,5\}$. Then $U_0 := \oplus_{i \in \mathbb{Z}_n} U(i;3)$ embeds into $M_J$ and is of the right dimension but 
\[ \dim_{\mathbb{C}} {\rm T}_{U_0}X_J(3,5) = \dim_\mathbb{C} {\rm Hom}_{\Delta_n}(U_0, N) = 8 > \dim_\mathbb{C} X_J(3,5) = 6,\]
where $N = U(1;2) \oplus U(2;2) \oplus U(3;1) \oplus U(4;2) \oplus U(5;1)$.
\end{rem}
To prove that the tangent space is isomorphic to an appropriate hom space, as in the previous remark, 
one can adapt the proof of  \cite[Proposition 6]{CaRe2008} by Caldero and Reineke. In fact, their arguments generalize in the same way as presented in \cite[Proposition~2.5]{Pue2019}.


\begin{rem}
In general, the cellular decomposition as in Theorem~\ref{trm:cellular-decomposition} is a refinement of the stratification based on \cite[Lemma~2.4]{CFR12}. It turns out that for the representation $U_{[n]}$, both decompositions coincide, as discussed below. 
In particular, we can assign  to each irreducible component of $X(k,n)$ a specific Grassmann necklace in $\mathcal{GN}_{k,n}$.
\end{rem}

\begin{thm}\label{trm:cells-are-strata-part-I}For $k,n \in \mathbb{N}$ with $k < n$ the following holds:  
\begin{enumerate}
 \item Two points of $X(k,n)$ belong to the same cell if and only if they are isomorphic as $\Delta_n$ modules.
    \item Each cell contains exactly one $\mathbb{C}^*$ fixed point.
    \item A cell equals the ${\rm Aut}_{\Delta_n}(U_{[n]})$ orbit of the $\mathbb{C}^*$ fixed point sitting in this cell. 
\end{enumerate}
\end{thm}

\begin{proof}
By construction each cell contains exactly one $\mathbb{C}^*$ fixed point. It follows from the parametrization of the $\mathbb{C}^*$ fixed points as in the proof of Theorem~\ref{trm:cellular-decomposition} that their corresponding coordinate subrepresentations of $U_{[n]}$ are pairwise non-isomorphic. This implies that the cells are the same as the strata of the fixed points, since in general each stratum in a quiver Grassmannian for a nilpotent representation of $\Delta_n$ decomposes into cells of isomorphic $\mathbb{C}^*$ fixed points. As $U_{[n]}$ is an injective bounded $\Delta_n$ representation (see Remark~\ref{rem:bound-quiver-relations}), we can prove analogous to \cite[Lemma~6.3]{Reineke2008} that the ${\rm Aut}_{\Delta_n}(U_{[n]})$ orbits are exactly the strata (see \cite[Lemma~2.28]{Pue2019}). 
\end{proof}

\begin{rem}
We can prove an analogous version of Theorem~\ref{trm:cells-are-strata-part-I} for the quiver Grassmannians ${\rm Gr}_{\bf k}(M_J)$ with $M_J := {\rm End}_{\Delta_n}(U_{[n]}).U_J$, which are isomorphic to the irreducible components of $X(k,n)$ (cf. Proposition~\ref{prop:irred-comp-are-quiver-Grass}). In particular, each cell is the ${\rm Aut}_{\Delta_n}(M_J)$ orbit of the corresponding fixed point. Note that $M_J$ is not an injective bounded $\Delta_n$ representation. 
Hence for the proof that strata are ${\rm Aut}_{\Delta_n}(M_J)$ orbits it is required to
explicitely extend any automorphism of $U \in {\rm Gr}_{\bf k}(M_J)$ to an automorphism of $M_J$. This can be done by exploiting that each  indecomposable summand of $M_J$ has multiplicity one, following an analogous procedure to \cite[Lemma~2.27]{Pue2019}.
\end{rem}

\section{Torus actions}\label{sec:T-action}
Analogous to \eqref{rem:C*-action}, we can construct actions of tori on the vector spaces of $M$, using multiple weight functions. Once again, these actions extend to the quiver Grassmannians of $M$ only under special assumptions. In this section we introduce several tori acting via weight tuples and  explain whether their actions extend to $X(k,n)$.
\begin{rem}\label{rem:n-dim-torus}
The ``obvious" torus of ${\rm Aut}_{\Delta_n}(U_{[n]})$ is only $n$-dimensional. Observe that this action extends to $X(k,n)$ but has infinitely many one-dimensional orbits in general, while we are interested in torus actions whose fixed point set and one-dimensional orbit set are finite.
\end{rem}

Let $M$ be a $\Delta_n$ representation and let $B$ be a basis of $M$. A weight tuple is a collection of integer valued vectors all of the same dimension: $\{\mathbf{wt}(b)=(w_1(b), \ldots ,w_r(b))\}_{b\in B}$, where $\mathbf{wt}(b)\in\bZ^r$ for some $r$. Given a weight tuple, we can define an action of a rank $r$ torus $T\simeq (\bC^*)^r$ on (the vector space) $M$ by setting
\[
(\gamma_1, \ldots, \gamma_r).b=\gamma_1^{w_1(b)}\cdot\ldots\cdot \gamma_d^{w_r(b)}b \qquad (b\in B, (\gamma_1, \ldots,\gamma_r)\in T).
\]
All torus actions we deal with are obtained by weight tuples.

\subsection{The largest Torus acting faithfully}
The largest torus acting faithfully on the vector spaces of $U_{[n]}=\bigoplus_{i \in \mathbb{Z}_n} U(i;n)$ is $n^2$-dimensional, since each vector space of $U_{[n]}$ over $i \in \mathbb{Z}_n$ is $n$-dimensional and we can act on each basis vector with a different parameter. But the maximal torus whose action extends to $X(k,n)$ is much smaller:
\begin{lem}
The action of $T' := (\mathbb{C}^*)^{r}$ on $X(k,n)$ with $r \geq 2n$ factors through the faithful action of a rank $2n$ torus on $X(k,n)$.
\end{lem}

\begin{proof}
Assume that we have a weight tuple such that the corresponding $T$ action extends to each quiver Grassmannian associated to $U_{[n]}$. Then this weight tuple has a fixed weight difference along each arrow of $\Delta_n$, i.e. if there are two arrows $b_1 \to b_1'$ and $b_2 \to b_2'$ in $Q(U_{[n]},B)$ with the same underlying arrow of $\Delta_n$, then  
\[\mathbf{wt}(b_1') - \mathbf{wt}(b_1) = \mathbf{wt}(b_2') - \mathbf{wt}(b_2)\]
(see the proof of \cite[Lemma~5.10]{LaPu2020}).
The above property of the weight tuples is equivalent to the condition that the grading is constructible from the weights of the predecessor free points in $Q(U_{[n]},B)$ and the weights of the edges of $\Delta_n$. For $U_{[n]}=\bigoplus_{i \in \mathbb{Z}_n} U(i;n)$ and $B$ being the standard basis for each copy of $\mathbb{C}^n$, the coefficient quiver $Q(U_{[n]},B)$ has exactly $n$ predecessor free points. This implies the claim since $\Delta_n$ has $n$ arrows. Hence we can choose at most $2n$ independent parameters.
\end{proof}

\begin{cor}
There exists a faithful $(\mathbb{C}^*)^{2n}$ action on $X(k,n)$.
\end{cor} 

This corollary  can also be obtained using the connection of $X(k,n)$ with the affine flag varieties, following Knutson-Lam-Speyer \cite{Kn08}.

\subsection{Skeletal Torus Action}\label{sec:Skeletal-T-Action}
In \cite{LaPu2020} we introduced an action of $T=(\mathbb{C}^*)^{n+1}$ on quiver Grassmannians for nilpotent representations of the equioriented cycle. We recall the action in the case of our distinguished quiver Grassmannian $X(k,n)$. First of all we enumerate the connected components (i.e. segments) of the coefficient quiver and hence denote by $s_j$ the (unique) segment ending in $j$,  that is corresponding to the indecomposable summand $U(j;n)$. We denote moreover by $b_{j,p}$ the basis vector of $U(j;n)$ corresponding to the $p$-th vertex of $s_j$ (e.g., the starting point of $s_j$ is denoted by $b_{j,1}$, while the end vertex is $b_{j,n}$). We define a $T$ action on the underlying vector space to $M$ by setting
\begin{equation}\label{eqn:skeletalTAction}
(\gamma_0, \gamma_1,\ldots,\gamma_n).b_{j,p}=\gamma_0^p\gamma_j b_{j,p} \qquad (\gamma_0,\gamma_1, \ldots, \gamma_n)\in T,\quad j,p\in [n]
\end{equation}
and then extending by linearity. By \cite[Lemma 5.10]{LaPu2020} this induces an action on $X(k,n)$. From now on we deal with this torus action.


We recall here the definition of skeletality, which is necessary in next section in order to consider moment graphs.
\begin{dfn}Let $T$ be a torus acting on a complex projective algebraic variety $X$.
The $T$-action on $X$ is said to be
  skeletal if the number of $T$-fixed points and one-dimensional $T$-orbits in $X$ is finite. 
\end{dfn}

We close this section with the following proposition.
\begin{prop}\label{trm:cells-are-strata-part-II} For $k,n \in \mathbb{N}$ with $k < n$ the following holds:  
\begin{enumerate}
    \item The action of $T$ on $X(k,n)$ is skeletal.
    \item Each cell (from Theorem~\ref{trm:cells-are-strata-part-I}) contains exactly one $T$ fixed point.
\end{enumerate}
\end{prop}

\begin{proof}
$T$ fixed points are the same as $\mathbb{C}^*$ fixed points since $T:= (\mathbb{C}^*)^{n+1}$ acts on $X(k,n)$ as in \cite[Lemma~5.10]{LaPu2020}. The action is skeletal since the number of $T$ fixed points is finite by \cite[Theorem~5.12]{LaPu2020} and the number of one-dimensional $T$ orbits is finite by \cite[Proposition~6.3]{LaPu2020}.
\end{proof}

\section{Moment Graph}\label{sec:moment-graph}
\subsection{The description}
If a complex algebraic variety $X$ is acted upon by a torus $T$ via a skeletal action, one can consider the corresponding moment graph. This is usually an unoriented graph, but if $X$ admits a $T$-stable cellular decomposition (as in our case) one can give the edges an orientation. More precisely:
\begin{dfn}Let $T$ be an algebraic torus and
let $X$ be a complex projective algebraic $T$-variety. Assume that $X$ admits a $T$-stable cellular decomposition where every cell has exactly one fixed point. If the action of $T$ on $X$ is skeletal, then the corresponding moment graph is given by 
\begin{itemize}
    \item the vertex set is the fixed point set: $\mathcal{V}=X^T$;
    \item there is an edge $x\to y$ if and only if $x$ and $y$ belong to the same one dimensional $T$ orbit closure $\overline{\mathcal{O}_{x\to y}}$ and $y$ belong to the closure of the cell containing $x$;
    \item the label of the edge $x\to y$ is the character $\alpha\in {\rm Hom}(T,\bC^*)$ the torus acts by on $\mathcal{O}_{x\to y}$. 
\end{itemize}
\end{dfn}

The label of any edge is only well defined up to a sign, but since this does not play any role in the applications (e.g. computation of equivariant cohomology), we assume the labels to be fixed once and for all, and forget about this ambiguity.

We want to explicitly describe the moment graph corresponding to the torus action on $X(k,n)$ from Section~\ref{sec:Skeletal-T-Action}. 

First of all, we need to relate the second basis of Definition \ref{dfn:different-relisation-of-U_[n]} to the basis $\{b_{j,p}\}_{j,p\in\bZ_n}$ of $U_{[n]}$ that we used to define the $T$ action \eqref{eqn:skeletalTAction}. It is immediate to see (by induction on $n-p$) that
\[
v_p^{(j+p)}=b_{j,p}, \quad \hbox{for all } j,p\in \bZ_n.
\]

For the rest of this section we consider $B=\bigcup_{i\in\bZ_n}\{v^{(i)}_1, \ldots, v_n^{(i)}\}$. Thus a successor closed subquiver $Q'$ of dimension $(k, \ldots, k)$ of $Q(U_{[n]},B)$ is a full subquiver whose vertex set is $Q'_0=Q'^{(1)}_0\bigsqcup Q'^{(2)}_0\bigsqcup\ldots\bigsqcup Q'^{(n)}_0$, with $Q'^{(i)}_0=\{v^{(i)}_{h_1},\ldots, v^{(i)}_{h_k}\}=Q'_0\cap B^{(i)}$ such that for any $i\in [n]$:
\[
Q'^{(i+1)}_0=\left\{
\begin{array}{ll}
\left\{v^{(i+1)}_{h_1+1},\ldots, v^{(i+1)}_{h_k+1}\right\}
 &\!\!\!\!\!\!\!\!\!\!\!\!\!\!\!\!\!\!\!\!\!\!\hbox{if }v_n^{(i)}\not\in Q'^{(i)}_0,\\
 \left\{v^{(i+1)}_{h_j+1} \Big\vert v^{(i)}_{h_j}\in Q'^{(i)}_0,\  h_j\neq n\right\}\cup\left\{ v_h^{(i+1)} \right\} \hbox{\small for some }h\in[n],&  \hbox{otw}.
\end{array}
\right.
\]
 
By Theorem \ref{trm:cellular-decomposition}, Proposition \ref{trm:cells-are-strata-part-II} and \cite[Proposition 1]{Cerulli2011}, both Grassmann necklaces and successor closed subquivers of $Q(U_{[n]},B)$ of dimension $(k,k,\ldots,k)$ parametrize the $T$ fixed point set. We write down here the explicit correspondence:
\begin{equation}\label{eqn:psi}
    \psi:\textrm{SC}_{\mathbf{k}}(U_{[n]})\rightarrow \mathcal{GN}_{k,n}\qquad
Q'\mapsto \left(\psi^{(a)}(Q'^{(a)}_0)\right)_{a\in [n]} 
\end{equation}
where $\psi^{(a)}(Q'^{(a)}_0)\in  {[n]\choose k}$ is the image of $Q_0'^{(a)}$ under the following map
\[
\psi^{(a)}:B^{(a)}\rightarrow [n] \quad v^{(a)}_h\mapsto a-h.
\]
It is easy to check that $\psi$ is well defined and bijective.

Recall from \S\ref{subsec:Grassmann-Necklace} that for any $a\in [n]$ there is a corresponding total order $\leq_a$ on $[n]$. The corresponding partial order on $\mathcal{GN}_{k,n}$ was denoted by $\leq$. For any $a,x,y\in [n]$, we define  
\[
[x,y]_a=\{z\in [n]\mid x\leq_a z\leq_a y \},
\]
and $[x,y)_a=[x,y]_a\setminus\{y\}$. 

The following notion allows us to provide a combinatorial description of the moment graph.
\begin{dfn}\label{dfn:mutation}Let $M\in {\textrm rep}_{\mathbb{C}}(\Delta_n)$, ${\bf e}\in\mathbb{Z}_{\geq 0}^n$ and $B=\cup_{i\in\mathbb{Z}_n}\{w^{(i)}_1, \ldots w^{(i)}_{m_i}\}$, and consider two successor closed subquivers $Q', Q''$ of $ Q(M, B)$. We say that $Q''$ is obtained from $Q'$ by a  mutation if 
\begin{itemize}
    \item there exists a segment $s_j\subseteq Q(M,B)$ such that 
    \[s_j\cap Q'= w^{(a)}_{p_1}\to w^{(a+1)}_{p_2}\to\ldots w^{(a+l-1)}_{p_l}\]
for some $a,l\in\mathbb{Z}_n$ and ${\mathbf{p}}\in[n]^l$,
\item there exist $l',r\in [n]$ with $l'\leq l$ and $p_i+r\leq n$ for any $i\in [l']$ such that
\begin{align*}
Q''&=\left(Q'\setminus\{w^{(a)}_{p_1}\to w^{(a+1)}_{p_2}\to\ldots w^{(a+l'-1)}_{p_{l'}}\}\right)\\
&\qquad\cup\{w^{(a)}_{p_1+r}\to w^{(a+1)}_{p_2+r}\to\ldots w^{(a+l'-1)}_{p_{l'}+r}\}.
\end{align*}
\end{itemize}
\end{dfn}

We are now ready to provide a description of the moment graph of the $T$-action on $X(k,n)$ which only involves Grassmann necklace combinatorics.

\begin{prop}\label{prop:MomentGraph}
The moment graph $\mathcal{G}$ of $X(k,n)$ has the following form:
\begin{enumerate}
    \item $\mathcal{G}_0=\mathcal{GN}_{k,n}$;
    \item $\mathcal{I}\rightarrow \mathcal{I}'$ if and only if $\mathcal{I}'<\mathcal{I}$ and there exist $a,b\in [n]$ such that 
    \begin{itemize}
        \item $I_h\neq I'_h$ if and only if $h\in [a,b]_a$,
        \item there exist $j,j'\in [n]$ such that   and   $I_h=(I_{h}'\setminus\{j'\})\cup \{j\}$ for any $h\in [n]$,
        \end{itemize}
    \item the label of  $\mathcal{I}\rightarrow \mathcal{I}'$ as above is $\epsilon_j-\epsilon_{j'}-\#[j',j)_a\delta$,
    where $\delta(\gamma)=\gamma_0$ and $\epsilon_i(\gamma)=\gamma_i$ for any $i\in [n]$, if $\gamma=(\gamma_0, \gamma_1, \ldots \gamma_n)\in T$.
\end{enumerate}
\end{prop}
\begin{proof}
By \cite[Theorem~6.13]{LaPu2020}, $\mathcal{G}_0$ can be identified with $\textrm{SC}_{\mathbf{k}}(U_{[n]})$ which in turn we identify with $\mathcal{GN}_{k,n}$ via the bijection $\psi$.
Again by \cite[Theorem~6.13]{LaPu2020}, there is an edge $Q'\rightarrow Q''$ for $Q',Q''\in\textrm{SC}_{\bf k}(M)$ if and only if $Q''$ is obtained from $Q'$ by a mutation, that is there exist $a,p,p',l\in[n]$ with $p<p'$ such that
\begin{itemize}
    \item $Q'\setminus (Q''\cap Q')=v^{(a)}_p\to v^{(a+1)}_{p+1}\to\ldots \to v^{(a+l-1)}_{p+l-1}
    $
     \item $Q''\setminus (Q'\cap Q'')=v^{(a)}_{p'}\to v^{(a+1)}_{p'+1}\to\ldots \to v^{(a+l-1)}_{p'+l-1}$.
\end{itemize}
Since $v_p^{(a)}\in s_{a-p}$ and $v_p'^{(a)}\in s_{a-p'}$, then \[
Q'\setminus (Q''\cap Q')\subseteq s_{a-p}, \quad Q''\setminus (Q''\cap Q')\subseteq s_{a-p'}.\] If we set $j:=a-p$ and $j':=a-p'$ the corresponding edge is labeled by $\epsilon_j-\epsilon_{j'}+(p-p')\delta$.  

Now we want to translate the conditions on $Q',Q''$ into conditions on the pair of Grassmann necklaces $\psi(Q')=\mathcal{I},\psi(Q'')=\mathcal{I}'$. Observe that $j'<_a j$ as $p'>p$, and set $b:=a+l-1$. Hence, we have that the above conditions on $Q$ and $Q'$ translate into 
\begin{itemize}
    \item $I_h=I'_h$ for all $h\not\in\{a,a+1,\ldots,a+l-1\}=[a,b]_{a}$,
    \item $I_h\neq I'_h$ and $I_h=(I_h'\setminus\{j'\})\cup\{j\}$ for all $h\in[a,b]_a$, 
\end{itemize}
where the second condition follows from the fact that
\[
\psi^{(a)}(v^{(a)}_p)=\psi^{(a+1)}(v^{(a+1)}_{p+1})=\psi^{(a+l-1)}(v^{(a+l-1)}_{p+l-1})=a-p=j\]
and
\[
\psi^{(a)}(v^{(a)}_{p'})=\psi^{(a+1)}(v^{(a+1)}_{p'+1})=\psi^{(a+l-1)}(v^{(a+l-1)}_{p'+l-1})=a-p'=j'.\]

Notice that $j'<_hj$ for any $h\in[a,b]_a$, thus $I'_h\leq_h I_h$ for any $h\in[n]$, that is $I'<I$. Finally, observe that $p-p'=-\#[j',j)_a$.
\end{proof}
\begin{example}\label{exple:MGX13} 
Let $n=3$ and $k=1$. We keep the same notation as in Example \ref{exple:PosetGN13} and write $i_1i_2i_3$ for the Grassmann necklace $(I_1=\{i_1\},I_2=\{i_2\},I_3=\{i_3\})\in\mathcal{GN}_{1,3}$. Moreover, to shorten notation, we write $\alpha_{i,j}:=\epsilon_i-\epsilon_j$ for $i,j\in[3]$.
Then by Proposition \ref{prop:MomentGraph}, the corresponding moment graph is 

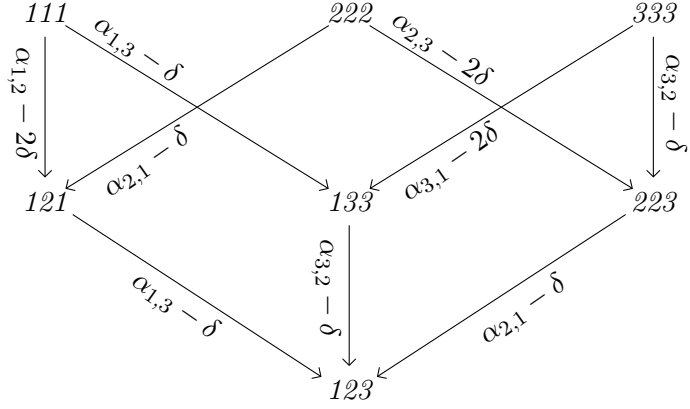
\begin{center}
\begin{tikzpicture}
\node at (0,0) {123};
\node at (0,2.5) {133};
\node at (0,5) {222};
\node at (-4,2.5) {121};
\node at (-4,5) {111};
\node at (4,2.5) {223};
\node at (4,5) {333};

\draw[arrows={-angle 90}, shorten >=9, shorten <=9]  (-3.9,2.5) -- (-.1,0);
\draw[arrows={-angle 90}, shorten >=9, shorten <=9]  (3.9,2.5) -- (.1,0);
\draw[arrows={-angle 90}, shorten >=9, shorten <=9]   (0,2.5) -- (0,0);

\draw[arrows={-angle 90}, shorten >=9, shorten <=9]  (-4,5) -- (-4,2.5);
\draw[arrows={-angle 90}, shorten >=9, shorten <=9]  (-4,5) -- (0,2.5);

\draw[arrows={-angle 90}, shorten >=9, shorten <=9]  (4,5) -- (4,2.5);
\draw[arrows={-angle 90}, shorten >=9, shorten <=9]  (4,5) -- (0,2.5);

\draw[arrows={-angle 90}, shorten >=9, shorten <=9]  (0,5) -- (4,2.5);
\draw[arrows={-angle 90}, shorten >=9, shorten <=9]  (0,5) -- (-4,2.5);

\node[rotate=325] at (-2.8,4.5) {$\alpha_{1,3}-\delta$};
\node[rotate=325] at (-2.3,1.1) {$\alpha_{1,3}-\delta$};
\node[rotate=325] at (1.2,4.5) {$\alpha_{2,3}-2\delta$};

\node[rotate=35] at (2.3,1.05) {$\alpha_{2,1}-\delta$};
\node[rotate=35] at (-2.65,3.) {$\alpha_{2,1}-\delta$};
\node[rotate=35] at (1.35,3.) {$\alpha_{3,1}-2\delta$};

\node[rotate=-90] at (-.3,1.35) {$\alpha_{3,2}-\delta$};
\node[rotate=-90] at (-4.3,3.8) {$\alpha_{1,2}-2\delta$};
\node[rotate=-90] at (4.25,3.8) {$\alpha_{3,2}-\delta$};

\end{tikzpicture}
\end{center}
Notice that the underlying graph coincides with the Hasse diagram we found in Example \ref{exple:PosetGN13}. This if of course not always the case, as the Hasse diagram is in general missing several of the edges of the moment graph.
\end{example}

\subsection{$T$-equivariant cohomology}\label{subsec:T-equiv-cohom}
By \cite[Corollary 5.13]{LaPu2020}, the moment graph described in the previous result encodes all needed information to compute the ($T$-equivariant) cohomology ring of $X(k,n)$. More precisely, let $R:=\mathbb{Q}[\epsilon_1, \ldots, \epsilon_n, \delta]$ and consider it as a $\mathbb{Z}$-graded ring with $\deg(\epsilon_i)=\deg(\delta)=2$ ($i\in [n]$). Denote by $\alpha(\mathcal{I},\mathcal{I}')$ the label of the edge $\mathcal{I}\to\mathcal{I}'$. \cite[Corollary 5.13]{LaPu2020} gives immediately the following result.
\begin{cor}\label{cor:GKM}
There is an isomorphism of ($\mathbb{Z}$-graded) rings
\[
H_T^*(X(k,n),\mathbb{Q})\simeq \left\{(z_{\mathcal{I}})\in\bigoplus_{\mathcal{I}\in\mathcal{GN}_{k,n}}R \, \Bigg\vert   \begin{array}{c}
  z_{\mathcal{I}}\equiv z_{\mathcal{I}'}\mod \alpha(\mathcal{I},\mathcal{I}')\\
\forall \hbox{ edge } \ \mathcal{I}\rightarrow\mathcal{I}'
\end{array} \right\}.
\]
\end{cor}

Moreover, by \cite[Theorem 3.21]{LaPu2021}, $H_T^\bullet(X(k,n))$ admits a very nice basis as a free module over $R$, namely a so-called Knutson-Tao basis (cf. \cite[Definition 3.2]{LaPu2021}). 

\begin{example}
By Corollary \ref{cor:GKM}, the $T$-equivariant cohomology of $X(1,3)$ can be read off from the moment graph from Example \ref{exple:MGX13}:
\[
H_T^*(X(1,3))\simeq \left\{(z_{123},z_{121}, z_{133},z_{223},z_{111},z_{222},z_{333}) \,\Bigg\vert \begin{array}{c}
  z_{i_1i_2i_3}\equiv z_{i_1'i_2'i_3'}\\\mod \alpha(i_1i_2i_3,i_1'i_2'i_3')\\
\forall \hbox{ edge } \ i_1i_2i_3\rightarrow i_1'i_2'i_3'
\end{array} \right\},
\]
where all $z_{i_1i_2i_3}\in R=\mathbb{Q}[\epsilon_1,\epsilon_2,\epsilon_3,\delta]$.
In this case the Knutson-Tao basis is
\[
(1,1,1,1,1,1,1), \ (0,0,\alpha_{3,2}-\delta,0,\alpha_{1,2}-2\delta,0,\alpha_{3,2}-\delta)
\]
\[
(0,\alpha_{1,3}-\delta,0,0,\alpha_{1,3}-\delta,\alpha_{2,3}-2\delta,0), \
(0,0,0,\alpha_{2,1}-\delta,0,\alpha_{2,1}-\delta,\alpha_{3,1}-2\delta),
\]
\[
(0,0,0,0,(\alpha_{1,2}-2\delta)(\alpha_{1,3}-\delta),0,0),\ (0,0,0,0,0,(\alpha_{2,1}-\delta)(\alpha_{2,3}-2\delta),0),
\]
\[
(0,0,0,0,0,0,(\alpha_{2,1}-\delta)(\alpha_{3,2}-\delta)).
\]
\end{example}

\section{Poincar\'e Polynomials}\label{sec:poincare-poly}

Recall from Section~\ref{sec:main-object} and Section~\ref{sec:T-action}:  $U(i;l)$  denotes the indecomposable $l$-dimensional $\Delta_n$ module terminating at the vertex $i$. The quiver Grassmannian $X(k,n)=\mathrm{Gr}_{\mathbf{k}}(U_{[n]})$, where  $U_{[n]}=\bigoplus_{i =1}^n U(i;n)$
admits a cellular decomposition. 
Each cell contains exactly one $T$ fixed point $p({\mathcal I})$, where 
${\mathcal I}=(I_1,\dots,I_n)$ is a $(k,n)$ Grassmann necklace and the $a$-th component $p({\mathcal I})_a$ of $p({\mathcal I})$ 
is a linear span of the 
basis vectors $e_i$, $i\in I_a$.  We denote the cell containing $p({\mathcal I})$ by $C({\mathcal I})$. 
We thus obtain the following formula for the Poincar\'e polynomial of the quiver Grassmannian $X(k,n)$:
\[
P_{k,n}(q)=\sum_{{\mathcal I}\in {\mathcal{GN}}_{k,n}} q^{\dim_\mathbb{C} C({\mathcal I})}.
\]


\begin{rem}
For any Grassmann necklace ${\mathcal I}$, the dimension of the cell $C({\mathcal I})$ equals the number of   outgoing edges from the corresponding vertex in the moment graph of $X(k,n)$ (cf. \cite[Corollary 6.4]{LaPu2020}).  There is a bijection between these edges and so called fundamental mutations of the $T$ fixed point corresponding to the Grassmann necklace ${\mathcal I}$ \cite[Theorem~6.13]{LaPu2020}. These combinatorial moves are essential for the proof of Lemma \ref{plus 1}.
\end{rem}

\begin{example}Let $n=3$ and $k=1$. 
By the previous remark, the Poincar\'e polynomial can be read off from the moment graph. Thus, from Example \ref{exple:MGX13} we deduce that
\[
P_{1,3}(q)=1+3q+3q^2.
\]
We notice that $P_{1,3}(q)$ is dual to the Poincar\'e polynomial of the $(1,3)$ tnn Grassmannian, which is $3+3q+q^2 = q^2P_{1,3}(q^{-1})$. This is a general fact (see Theorem \ref{PP}).
\end{example}

Let $\preceq$ be the partial order on $\mathcal{GN}_{k,n}$ coming from the cell closure relation in $X(k,n)$, and recall the partial order $\leq$ on $\mathcal{GN}_{k,n}$ introduced in \S\ref{subsec:Grassmann-Necklace}.
\begin{prop}\label{prop:posetIso}The partial orders $\preceq$ and $\leq$ on $\mathcal{GN}_{k,n}$ coincide.
\end{prop}
\begin{proof}
We show that the map $\psi$ from \eqref{eqn:psi} is a poset isomorphism, where we also use the notation $\preceq$ to denote the cell closure relation order on ${\rm SC}_{\mathbf{k}}(U_{[n]})$.

Since $\preceq$ is in our case generated by the mutations, it follows immediately from the proof of Proposition \ref{prop:MomentGraph} that $Q''\preceq Q'$ implies $\psi(Q'')\leq \psi(Q')$.

Assume now that $\mathcal{I},\mathcal{I}'\in \mathcal{GN}_{k,n}$ are such that $\mathcal{I}'\leq \mathcal{I}$, and denote $Q':=\psi^{-1}(\mathcal{I})$, $Q'':=\psi^{-1}(\mathcal{I}')$. For any $a\in[n]$ let
\[
{\mathbf h}^{(a)}:=(h_1^{(a)}<h_2^{(a)}<\ldots< h_k^{(a)}), \quad 
{\mathbf l}^{(a)}:=(l_1^{(a)}<l_2^{(a)}<\ldots< l_k^{(a)})\]
be such that 
\[
Q''_0\cap B^{(a)}=\left\{v^{(a)}_{h_j^{(a)}}\mid j\in[k]\right\}, \quad Q'_0\cap B^{(a)}=\left\{v^{(a)}_{l_j^{(a)}}\mid j\in[k]\right\}.
\]
By definition of $\leq$ and $\psi$, we have that \[
h_j^{(a)}\geq l_j^{(a)} \quad\hbox{for all }j\in[k], \ a\in[n].
\]
Thus, we want to show that there exists a sequence of mutations from $Q'$ to $Q''$. We proceed by induction on $d=n-\#\{a\mid \mathbf{h}^{(a)}=\mathbf{l}^{(a)}\}$. 

The base case is $\mathbf{h}^{(a)}=\mathbf{l}^{(a)}$ for all $a\in[n]$, that is $Q'=Q''$, and there is nothing to be shown. Otherwise, there exists an $\overline{a}\in[n]$ such that $\mathbf{h}^{(\overline{a})}\neq\mathbf{l}^{(\overline{a})}$. We hence find an $r\in[k]$ such that $h_j^{(\overline{a})}=l_j^{(\overline{a})}$ for all $j\in[r+1,k]$ and $h_r^{(\overline{a})}>l_r^{(\overline{a})}$. To save notation we write $h$ and $l$ for $h_r^{(\overline{a})}$ and $l_r^{(\overline{a})}$, respectively. 
Define
\[
\overline{c}:=\min_{\overline{a}-l+1}\left\{c\mid  v^{(c)}_{l+c-\overline{a}}\in Q'_0\cap B^{(c)}\right\}
, \quad 
\overline{e}:=\min_{\overline{a}-l+1}\left\{e\mid v^{(e)}_{h+e-\overline{a}}\in Q'_0\cap B^{(e)}\right\}.
\]
Observe that $v^{(\overline{a} )}_h=b_{\overline{a}-h,h}$, so that  
$v^{(\overline{c}+t-1)}_{h+\overline{c}+t-\overline{a}-1}\not\in Q'_0\cap B^{\overline{c}+t-1}$ for any $t\in[\overline{e}-\overline{c}]_{\overline{c}}$.
Thus, the following successor closed subquiver is obtained from $Q'$ by mutation:
\begin{align*}
Q'''&:=\left(Q'\setminus\left\{v_{l+\overline{c}-\overline{a}}^{(\overline{c})}\to \ldots \to v^{(\overline{a})}_{l}\to\ldots \to v^{(\overline{e}-1)}_{l+\overline{e}-\overline{a}-1}\right\}\right)\\
&\qquad\cup\left\{v_{h+\overline{c}-\overline{a}}^{(\overline{c})}\to \ldots \to v^{(\overline{a})}_{h}\to\ldots \to v^{(\overline{e}-1)}_{h+\overline{e}-\overline{a}-1}\right\}
\end{align*}
Note that, thanks to the choice of $\overline{c}$, we have $Q''\preceq Q'''\prec Q'$. 
Let 
\[
{{\bf l}'}^{(a)}={l'}_{1}^{(a)}<{l'}_{2}^{(a)}<\ldots<{l'}_{k}^{(a)}\ \ \hbox{be such that}\ \ Q_0'''\cap B^{(a)}=\left\{v_{{l'}^{(a)}_j}^{(a)} \mid j\in[k]\right\}, 
\]
then $h^{(\overline{a})}_j={l'}_j^{(\overline{a})}$ for any $j\in[r,k]$, and ${\bf h}^{(a)}={{\bf l}'}^{(a)}$ whenever ${\bf h}^{(a)}={\bf l}^{(a)}$. If $r=1$ we are done, otherwise we proceed recursively until we get equality for any $j\in[k]$ and we denote the resulting successor closed subquiver by $\overline{Q}$. By construction now 
\[
n-\#\{a\mid Q''_0\cap B^{(a)}=\overline{Q}_0\cap B^{(a)}\}<d
\]
and we can apply the inductive step to complete the proof.
\end{proof}

\begin{lem}\label{plus 1}
Assume that for two Grassmann necklaces ${\mathcal I}$ and ${\mathcal J}$ the closure of the cell
$C({\mathcal I})$ contains the cell $C({\mathcal J})$ and there is no cell $C$ such that
\[
\overline{C({\mathcal I})}\supset C,\qquad \overline{C}\supset C({\mathcal J}).
\]
Then $\dim_\mathbb{C} C({\mathcal I}) = \dim_\mathbb{C} C({\mathcal J}) + 1$.
\end{lem}

\begin{proof}
By Proposition~\ref{prop:MomentGraph} and Proposition~\ref{trm:cells-are-strata-part-II} every cell $C(\mathcal{I})$ of $X(k,n)$ contains a unique $T$ fixed point $p(\mathcal{I})$ together with a unique subquiver $S(\mathcal{I}) \subset Q(U_{[n]},B)$ consisting of segments $s_1, \dots, s_n$. Now the condition that there is no cell between $C(\mathcal{J})$ and $C(\mathcal{I})$ implies that there is a corresponding fundamental mutation $\mu: p(\mathcal{I}) \to p(\mathcal{J})$ and it does not factor through any other mutation. On the coefficient quivers $S(\mathcal{I})$ and $S(\mathcal{J})$ we can view this mutation as cutting exactly one predecessor closed subquiver of a segment $s_{j_1}$ of $S(\mathcal{I})$ and adding it to the start of a segment $s_{j_1}$ of $S(\mathcal{I})$ to obtain the segments $s_{j_1}'$ and $s_{j_2}'$ of $S(\mathcal{J})$ (see \cite[Definition~2.15]{LaPu2021} for more details). 

By convention the index of the element of $B$ corresponding to the starting point of $s_{j_1}$ is larger than the index of the basis element corresponding to the start of $s_{j_2}'$, i.e. mutations are always index increasing and hence we also speak of a downwards movement of the subsegment. All the other segments remain unchanged and hence coincide for both points. 

The height $h_\mathcal{I}(s_{j})$ is defined as the number of the mutations of $p(\mathcal{I})$ which start at the segment $s_{j}$. For $U_{[n]}$ this equals the number of points in $Q(U[n],B)$ which have larger index than the start of  $s_{j}$ and are not contained in $S(\mathcal{I})$. We sometimes refer to this as counting the holes below the start of $s_{j}$ (cf. \cite[Remark~6]{CFFFR2017}).

The starting points of the segments between $s_{j_1}$ and $s_{j_2}$ can not live over the support of the moved subsegment in $\Delta_n$ because otherwise $\mu$ would factor through them and this implies that they contribute as holes for the height functions $h_\mathcal{I}(s_{j_1})$ and $h_\mathcal{J}(s_{j_1}')$. Hence we obtain $h_\mathcal{J}(s_{j_1}') = h_\mathcal{I}(s_{j_1})-1$.

Every mutation starting at $s_{j_2}$ can also be started at $s_{j_2}'$ so we immediately obtain $h_\mathcal{J}(s_{j_2}') \geq h_\mathcal{I}(s_{j_2})$. For the other segments we distinguish two cases: If the segment $S_j$ is between $s_{j_1}$ and $s_{j_2}$, by the discussion above, its starting point can not be above the support of the moved segment. Hence both $s_{j_1}$ and $s_{j_2}$ have marked points in $Q(U[n],B)$ below the starting point of $s_j$ such that they both contribute as zero to the height functions $h_\mathcal{I}(s_{j})$ and $h_\mathcal{J}(s_{j})$. 

In the other case the start of the segment $s_j$ is either below the points of $s_{j_2}'$ and the height function does not change at all, or above $s_{j_1}$ and in this case the value of the height function stays the same since the role of the points on $s_{j_1}$ is exchanged with the points on $s_{j_2}'$ and vice versa for $s_{j_2}$ and $s_{j_1}'$.
By \cite[Proposition~3.19]{LaPu2021} we obtain $\dim_\mathbb{C} C(\mathcal{J}) \leq \dim_\mathbb{C} C(\mathcal{I}) -1$ and the above computation shows that equality is achieved in our setting.
\end{proof}


\begin{thm}\label{trm:cell-dim-quiver-Grass}
The dimension of the cell $C({\mathcal I})$ is equal to the length of the bounded affine
permutation $f({\mathcal I})$.
\end{thm}

\begin{proof}
By Proposition \ref{prop:posetIso}, we know that the poset structure of ${\mathcal{GN}}_{k,n}$ coming from the closure relations in $X(k,n)$ is isomorphic to the
poset $\mathcal{B}_{k,n}$ (see \cite{KLS14,Lam16}). 
We note that both posets contain a unique minimal element -- a zero-dimensional cell corresponding to
the collection $(I_j)_j$ with $I_j=\{j,j+1,\dots,j+k-1\}$ (all the numbers are taken mod $n$) and the length zero element
${\rm id}_k$. Since the set ${\mathcal B}_{k,n}$ of bounded affine permutations can be identified with the lower order ideal
in the affine Weyl group, we conclude that if $f>g$ with no element $h\in {\mathcal B}_{k,n}$ in between one has $l(f)=l(g)+1$. 
Now Lemma \ref{plus 1} completes the proof. 
\end{proof}

\begin{example}
If $\mathcal{I}=\{\{1,3\},\{3,4\},\{3,4\},\{1,4\}\}\in\mathcal{GN}_{2,4}$, we then have
$\dim_\mathbb{C}(C(\mathcal{I})) =2$.
\end{example}
The next result follows immediately from Theorem~\ref{trm:cell-dim-quiver-Grass} and Remark~\ref{rem:cell-dim-tnn-Grass}.
\begin{thm}\label{PP}
The Poincar\'e polynomial of $X(k,n)$ is dual to the Poincar\'e polynomial of the tnn Grassmannian. 
\end{thm}

\begin{rem}
In contrast to \cite[Remark~5.14]{LaPu2021}, Theorem~\ref{trm:cells-are-strata-part-I} implies that the closure of every cell in $X(k,n)$ is the union of smaller cells. Moreover by Proposition~\ref{prop:MomentGraph} we obtain it explicitly as
\[ \overline{C({\mathcal I})} =  \bigcup_{{\mathcal J}\in {\mathcal{GN}}_{k,n} \ {\rm s.t.:}  \ {\mathcal J} \leq {\mathcal I}} C({\mathcal J}).\]
The moment graph of $\overline{C({\mathcal I})}$ is the full subgraph of $\mathcal{G}$ from Proposition~\ref{prop:MomentGraph} on the vertices ${\mathcal J}\in {\mathcal{GN}}_{k,n}$ with ${\mathcal J} \leq {\mathcal I}$ and the computation of the corresponding $T$-equivariant cohomology and the Poincar\'e polynomial is done in the same way.
\end{rem}

\section{Resolution of singularities}\label{sec:desingularization}
The goal of this section is to present a construction of a  resolution of singularities for the irreducible
components of $X(k,n)$. In short, an irreducible component of $X(k,n)$ will be desingularized by a
quiver Grassmannian for the extended cyclic quiver $\tilde\Delta_n$ (see the definition below). 
For various constructions of desingularizations in the context of quiver
Grassmannians see \cite{CFR13,CFR14,FF13,FFL14,KS14,S17}. 

\subsection{The extended cyclic quiver and extended representations}
Let us denote by $\widetilde\Delta_n$ the extended cyclic quiver defined by the following data:
\begin{itemize}
    \item the vertices of $\widetilde\Delta_n$ are of the form $(i,r)$ with $i\in\bZ_n$, $r=1,\dots,n$,
    \item there are two types of arrows: $\alpha_{i,r}:(i,r)\to (i+1,r+1)$, $r<n$ 
    and $\beta_{i,r}: (i,r)\to (i,r-1)$, $r>1$.
\end{itemize}
\begin{example}
\begin{enumerate}
    \item The extended cyclic quiver for $n=2$ looks as follows:
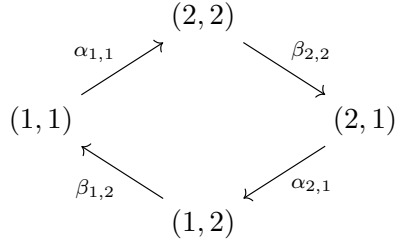
\[
\begin{tikzcd}
&  (2,2) \ar[dr, "\beta_{2,2}"] &\\
(1,1) \ar[ur, "\alpha_{1,1}"] &  & (2,1) \ar[dl, "\alpha_{2,1}"]\\
&  (1,2) \ar[ul, "\beta_{1,2}"]&
\end{tikzcd}
\]
    \item 
The extended cyclic quiver for $n=3$ looks as follows:
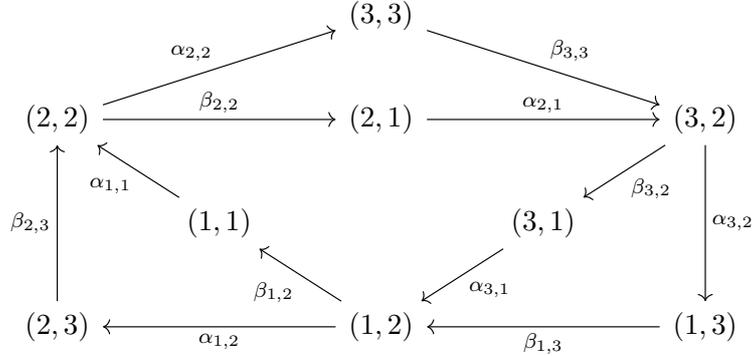
\[
\begin{tikzcd}
  &  & (3,3) \ar[drr, "\beta_{3,3}"] & &  \\
 (2,2) \ar[urr, "\alpha_{2,2}"] \ar[rr, "\beta_{2,2}"] &  &  (2,1) \ar[rr, "\alpha_{2,1}"] & &  (3,2) \ar[dd, "\alpha_{3,2}"] \ar[dl, "\beta_{3,2}"] \\
 & (1,1) \ar[ul, "\alpha_{1,1}"] & &   (3,1)\ar[dl, "\alpha_{3,1}"] &  \\
(2,3) \ar[uu, "\beta_{2,3}"]  & & (1,2) \ar[ll, "\alpha_{1,2}"] \ar[ul, "\beta_{1,2}"]  & &   (1,3) \ar[ll, "\beta_{1,3}"]
\end{tikzcd}
\]

\end{enumerate}
\end{example}

Now let $M=((M^{(i)})_{i\in \mathbb{Z}_n}, (M_{\alpha_i})_{i\in \bZ_n})$ be a $\Delta_n$ representation.  
 We define the extended $\widetilde \Delta_n$
module  $\widetilde M$  as follows. Let $M^{(i,1)}=M^{(i)}$ for $i\in\mathbb{Z}_n$ and let
\[
\widetilde M^{(i,r)}= M_{\alpha_{i-1}}\dots M_{\alpha_{i-r+1}} M^{(i-r+1)} \qquad (r=2,\dots,n).
\]
In particular, $\widetilde M^{(i,r)}\subset M^{(i)}$ for all $r$. We also have natural surjections and inclusions \begin{gather}
\widetilde M^{(i,n)}\twoheadleftarrow \widetilde M^{(i-1,n-1)}\twoheadleftarrow \dots\twoheadleftarrow \widetilde M^{(i-n+1,1)} = M^{(i-n+1)}\label{surj},\\  
\widetilde M^{(i,n)}\subset \widetilde M^{(i,n-1)} \subset \dots\subset \widetilde M^{(i,1)} = M^{(i)}\label{incl}.
\end{gather}
We complete the definition of the extended representation $\widetilde M$ by letting
\[
\widetilde M_{\alpha_{i,r}}:  \widetilde M^{(i,r)} \to \widetilde M^{(i+1,r+1)},\quad
\widetilde M_{\beta_{i,r}}:  \widetilde M^{(i,r)} \to \widetilde M^{(i,r-1)}
 \]
to be the maps from \eqref{surj} and \eqref{incl} respectively.

\begin{rem}\label{recover}
Let us identify $\widetilde M^{(i,1)}$ with $M^{(i)}$. Then the maps $M_{\alpha_i}$ are recovered as
compositions $\widetilde M_{\beta_{i+1,2}}\widetilde M_{\alpha_{i,1}}$.
\end{rem}

\begin{example}
Let $n=2$ and let $M=U(1;2)\oplus U(2;2)\in \mathrm{rep}_\mathbb{C}(\Delta_2)$. Then the $\widetilde\Delta_2$ module
$\widetilde M$ looks as follows
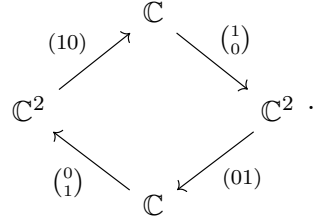
\[
\begin{tikzcd}
&  \bC \ar[dr, "\binom{1}{0}"] &\\
\bC^2 \ar[ur, "(1 0)"] &  & \bC^2 \ar[dl, "(0 1)"]\\
&  \bC \ar[ul, "\binom{0}{1}"] &
\end{tikzcd}.
\]
\end{example}

\begin{example}
Let $M=U_{[n]}=\bigoplus_{i=1}^n U(i;n)$, so that $\dim_\bC \widetilde M^{(i,r)}=n-r+1$. To describe the 
maps between the vector spaces $\widetilde M^{(i,r)}$ we use the second basis from Definition \ref{dfn:different-relisation-of-U_[n]}
Then $\widetilde M^{(i,r)}$ is identified with the linear span of vectors
$v^{(i)}_h$ with $h=r, \dots,n$ and the maps $\widetilde M_{\alpha_{i,r}}$ and $\widetilde M_{\beta_{i,r}}$ are given by
\[
\widetilde M_{\beta_{i,r}} v^{(i)}_h = v^{(i)}_h,\ \widetilde M_{\alpha_{i,r}} v^{(i)}_h = v^{(i+1)}_{h+1}, \text{ if } h<n,\ \widetilde M_{\alpha_{i,r}} v^{(i)}_n =0.
\]
\end{example}

\begin{lem}\label{lem:bounding-relations-part-1}
For a $\Delta_n$ module $M$ the extended representation $\widetilde M$ is a representation of the bounded quiver $\widetilde \Delta_n$ with the relations $\beta_{i+1,r+1}\alpha_{i,r} =\alpha_{i,r-1} \beta_{i,r}$
for all $i\in\bZ_n$ and $r=2,\dots,n-1$.
\end{lem}
\begin{proof}
Obvious from the definition of $\widetilde M$. 
\end{proof}

We say that an element of the path algebra of $\widetilde \Delta_n$ is an $\alpha$ path, respectively $\beta$ path, if it is given by the concatenation of maps exclusively  of type $\widetilde M_{\alpha_{\bullet}}$, respectively $\widetilde M_{\beta_\bullet}$.
\begin{lem}\label{lem:bounding-relations-part-2}
Assume that  the length $n$ paths of the path algebra of $\Delta_n$ act trivially on a $\Delta_n$ module $M$.
Then all the length $2n-1$ paths from the path algebra of $\widetilde\Delta_n$ act trivially on $\widetilde M$.
\end{lem}
\begin{proof}Let $i,r\in\bZ_n$ and let us take $v\in \widetilde M^{(i,r)}$. We attach to $v$ a pair $(r', r'')$, where $r'$ is the maximal length of a $\beta$ path on $\widetilde\Delta_n$ starting in $(i,r)$ and where $r''$ is the maximal length of an $\alpha$ path on $\widetilde\Delta_n$ ending in $(i,r)$. It follows immediately from the definition of $\widetilde M^{(i,r)}$ that $r'=r-1$.


Notice that
\begin{itemize}
    \item application of  a map $\widetilde M_{\beta_\bullet}$ to $v$ changes the coordinates $(r',r'')$ to $(r'-1,r'')$,
    \item application of a map $\widetilde M_{\alpha_\bullet}$ to $v$ changes the coordinates $(r',r'')$ to $(r'+1,r''+1)$,
    \item there are no vectors with coordinates $(r',r'')$ with $r'<0$ or $r''\ge n$. 
\end{itemize}
For example, $r''<n$, because any length $n$ path acts trivially on $M$. Our lemma is equivalent to the statement that
one can not do more than $2n-1$ steps of the form  $(r',r'')\to (r'-1,r'')$ or $(r',r'')$ to $(r'+1,r''+1)$ starting
at $(r-1,r-1)$ and staying inside the region $r'\ge 0$, $r''<n$. 
\end{proof}

\subsection{Construction of the desingularization}

\begin{lem}
Let $M\in \mathrm{rep}_\mathbb{C}(\Delta_n)$ 
let $U\subset M$ be an ${\bf e}$-dimensional subrepresentation. 
Then there exists a natural map
\begin{equation}\label{des}
\Psi: {\rm Gr}_{{\bf dim} \, \widetilde U}(\widetilde M) \to {\rm Gr}_{{\bf e}}(M),
\end{equation}
defined by the following rule: for $N\in {\rm Gr}_{{\bf dim} \, \widetilde U}(\widetilde M)$ its image 
$\Psi(N)\in {\overline S_U}$ is given by
\[
 \Psi(N)^{(i)}=N^{(i,1)}, \ \Psi(N)_{\alpha_i} = N_{\beta_{i+1,2}}N_{\alpha_{i,1}}.
\]
\end{lem}
\begin{proof}
By construction, $\Psi(N)$ is an ${\bf e}$-dimensional submodule of $M$ (see Remark \ref{recover}).
\end{proof}

Recall that the irreducible components of $X(k,n)$ are labeled by the $k$-element subsets 
$J\subset \{1,\dots,n\}$ and that we denote the corresponding irreducible component by $X_J(k,n)$ (see Remark~\ref{strata} and Proposition~\ref{prop:irred-comp-are-quiver-Grass}).
    
\begin{rem}
For any $J\subset \binom{[n]}{k}$ the irreducible component $X_J(k,n)$ contains a point $p_J$ which is a summand wise embedding of the representation 
 \[ U_J=\bigoplus_{j\in J} U(j;n) \] 
 into $U_{[n]}=\bigoplus_{i=1}^n U(i;n)$. The  component $X_J(k,n)$ is the closure of the ${\rm Aut}_{\Delta_n}(U_{[n]})$ orbit $C(p_J)$ of the point $p_J$. Each point of the cell $C(p_J)$ is isomorphic to $U_J$ as a $\Delta_n$ module and points in different cells are not isomorphic. 
 Moreover, each cell contains exactly one $\mathbb{C}^*$ fixed point (see Theorem~\ref{trm:cells-are-strata-part-I}).
\end{rem} 

Now let us consider the $\widetilde \Delta_n$ module $\widetilde U_J$ and let  ${\bf d}_J={\bf dim} \, \widetilde U_J$,
so that ${\bf d}_J$ is a vector with components $d_J^{(i,r)}$.

\begin{lem}
The numbers $d_J^{(i,r)}$ are explicitly given by
\[
d_J^{(i,r)} =  \# \left(J\cap \{i':\ i\le  i'\le i+n-r\}\right).
\]
In particular, $d_J^{(i,r)}\geq d_J^{(i,r+1)}$ for any $i\in\bZ_n$ and  $r\in[n-1]$.
\end{lem}
\begin{proof}
By definition, 
\[
d_J^{(i,r)} = \dim_\bC U_{J,\alpha_{i-1}}\dots U_{J,\alpha_{i-r+1}} U_J^{(i-r+1)}.
\]
One easily sees that
\[
U_{J,\alpha_{i-1}}\dots U_{J,\alpha_{i-r+1}} U(i';n)^{(i-r+1)} 
\]
is nontrivial (i.e. one-dimensional) if and only if  $i\le  i'\le i+n-r$.
\end{proof}

\begin{example}
One has $d_J^{(i,1)}=k$ for all $i$. The dimensions $d_J^{(i,n)}$ are either zero or one. 
\end{example}

\begin{thm}\label{issmooth}
The quiver Grassmannian ${\rm Gr}_{{\bf d}_J} (\widetilde U_{[n]})$ is smooth.
\end{thm}
\begin{proof}
We show that ${\rm Gr}_{{\bf d}_J} (\widetilde U_{[n]})$ is isomorphic to a tower of fibrations 
\[
{\rm Gr}_{{\bf d}_J} (\widetilde U_{[n]})=Y_1\to Y_2\to\dots\to Y_n={\rm pt}
\]
such that each map $Y_r\to Y_{r+1}$, $r=1,\dots,n-1$ is a fibration with the  fiber being a product of classical Grassmann varieties. 

A point $N$ of ${\rm Gr}_{{\bf d}_J} (\widetilde U_{[n]})$ is a collection of subspaces 
$N^{(i,r)}\subset \widetilde U_{[n]}^{(i,r)}$. 
We define $Y_r$ as the image of the natural projection map
\[
{\rm Gr}_{{\bf d}_J} (\widetilde U_{[n]})\to \prod_{r'=r}^n \prod_{i\in\bZ_n} {\rm Gr}_{d_J^{(i,r')}}(\widetilde U_{[n]}^{(i,r')}), \quad N\mapsto (N^{(i,r)})_{\substack{i\in\bZ_n\\r\leq r'\leq n}}.
\]

We start with $r=n$ and then proceed by decreasing induction on $r$.
One has $\dim_\bC \widetilde{U}_{[n]}^{(i,n)}=1$ and $d_J^{(i,n)}$ are either 
$0$ or $1$. Thus $Y_n$ is a product of $n$ points (corresponding to $i\in\bZ_n$).  We also note 
 that
\[
U_{\alpha_{i,n-1}}U_{\beta_{i,n}}  (N^{(i,n)}) \subset N^{(i+1,n)},
\]
because the composition $U_{\alpha_{i,n-1}}U_{\beta_{i,n}}$ vanishes on $\widetilde U_{[n]}$. Here and below we use the notation 
$U_\gamma=\widetilde{U}_{[n],\gamma}$ for the map corresponding to the edge $\gamma$ of $\widetilde \Delta_n$. 

Now assume that all the subspaces $N^{(i,r')}$ are fixed for $r'>r$. 
Since $N$ is a subrepresentation, 
 the subspace $N^{(i,r)}\subset \widetilde{U}_{[n]}^{(i,r)}$ has to satisfy the conditions
\[
U_{\beta_{i,r+1}} N^{(i,r+1)}\subset N^{(i,r)},\ U_{\alpha_{i,r}} N^{(i,r)} \subset N^{(i+1,r+1)}  
\]
and hence 
\begin{equation}\label{condition}
U_{\alpha_{i,r}} U_{\beta_{i,r+1}}  (N^{(i,r+1)}) \subset N^{(i+1,r+1)}.
\end{equation}
Since $U_{\beta_{i,r+1}}$ is an embedding, $U_{\alpha_{i,r}}$ is a surjection and condition \eqref{condition} holds, the choice of $N^{(i,r)}$ as above is equivalent to the choice of a point
in the Grassmannian 
\[
{\rm Gr}_{d_J^{(i,r)}-d_J^{(i,r+1)}}\bigl(\widetilde U_{[n]}^{(i,r)}/U_{\beta_{i,r+1}}(N^{(i,r+1)})\bigr).
\]
\end{proof}

\begin{prop}
The image of the natural map $f_J: {\rm Gr}_{{\bf d}_J} (\widetilde U_{[n]})\to X(k,n)$ is equal to
$X_J(k,n)$ and is a bijection over the open cell $C(p_J)$. 
\end{prop}
\begin{proof}
Let us take a point $M\in C(p_J)$. By Theorem \ref{trm:cells-are-strata-part-I}(i), all the points in the cell are isomorphic as $\Delta_n$ modules and it follows that
\[
\dim_\bC M_{\alpha_{i-r+1}}\dots M_{\alpha_i}M^{(i)} = d_J^{(i,r)} 
\]
for all vertices $(i,r)$. Hence the preimage of $M$ is a single point. Moreover, we claim that $f_J^{-1}(C(p_J))$ is open in ${\rm Gr}_{{\bf d}_J}(\widetilde M)$. The proof is analogous to the proof of Theorem \ref{issmooth}: we consider the tower
\[
f_J^{-1}(C(p_J))=Y_1\cap f_J^{-1}(C(p_J))\to Y_2\cap f_J^{-1}(C(p_J))\to\dots\to Y_n\cap f_J^{-1}(C(p_J))={\rm pt}
\]
and show by decreasing induction on $r$ that at each step one obtains an open part of $Y_r$.

Recall that by definition $X_J(k,n)$ coincides with the closure
of the cell $C(p_J)$.  Therefore
\[
f_J\bigl({\rm Gr}_{{\bf d}_J} (\widetilde U_{[n]})\bigr) = 
\overline{f_J(f_J^{-1}C(p_J))}=\overline{C(p_J)}=X_J(k,n), 
\]
where the first equality is true since $f_J^{-1}(C(p_J))$ is open in ${\rm Gr}_{{\bf d}_J}(\widetilde M)$. 
\end{proof}

We summarize the whole picture in the following Corollary.
\begin{cor}\label{cor:desing}
The quiver Grassmannian ${\rm Gr}_{{\bf d}_J}(\widetilde U_{[n]})$ desingularizes $X_J(k,n)$.
\end{cor}

\subsection{Example}
We work out two examples of the desingularization above for $n=4$ and $k=2$. The variety $X(2,4)$ has six irreducible 
components,  labeled by the cardinality two sets $J\subset [4]$. Essentially there are two different cases (up to rotation): $J=\{1,2\}$ and $J=\{1,3\}$. The quiver
$\widetilde \Delta_4$  looks as follows:
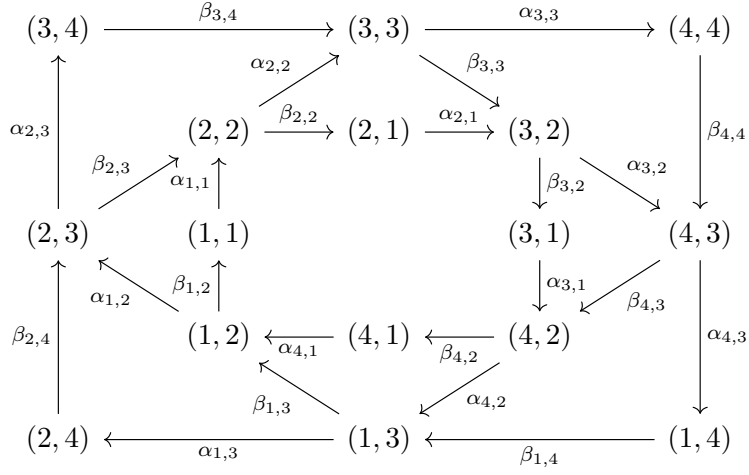
\[
\begin{tikzcd}
(3,4) \ar[rr, "\beta_{3,4}"] & & (3,3) \ar[dr, "\beta_{3,3}"]\ar[rr, "\alpha_{3,3}"] &  &  (4,4) \ar[dd, "\beta_{4,4}"]  \\
 & (2,2) \ar[r, "\beta_{2,2}"] \ar[ur, "\alpha_{2,2}"]  & (2,1) \ar[r, "\alpha_{2,1}"] & (3,2) \ar[d, "\beta_{3,2}"] \ar[dr, "\alpha_{3,2}"] &  \\
 (2,3) \ar[uu, "\alpha_{2,3}"] \ar[ur, "\beta_{2,3}"]  & (1,1) \ar[u, "\alpha_{1,1}"]  & & (3,1) \ar[d, "\alpha_{3,1}"]  & (4,3) \ar[dl, "\beta_{4,3}"]\ar[dd, "\alpha_{4,3}"]\\
 & (1,2) \ar[ul, "\alpha_{1,2}"] \ar[u, "\beta_{1,2}"] & (4,1) \ar[l, "\alpha_{4,1}"] & (4,2) \ar[l, "\beta_{4,2}"] 
 \ar[dl, "\alpha_{4,2}"] &\\
 (2,4) \ar[uu, "\beta_{2,4}"] & & (1,3) \ar[ll, "\alpha_{1,3}"] \ar[ul, "\beta_{1,3}"] & & (1,4) \ar[ll, "\beta_{1,4}"]
\end{tikzcd}
\]

Now let $J=\{1,2\}$. Then the dimension vector $d_J^{(i,r)}$ is given by
\medskip
\[
\begin{tikzcd}
0 & & 0 &  &  0 \\
 & 1 & 2  & 1 &  \\
1  & 2 & & 2 & 1 \\
 & 2  & 2 & 2 &\\
 1 & & 2 & & 1
\end{tikzcd}
\]
\medskip
Hence the desingularization ${\rm Gr}_{{\bf d}_J} (\widetilde U_{[n]})$ is the tower bundle 
\begin{equation}\label{tower}
{\rm Gr}_{{\bf d}_J} (\widetilde U_{[n]}) = Y_1 \to Y_2 \to Y_3 \to Y_4=pt,  
\end{equation}
where the fiber of the fibration $Y_3\to Y_4$ is $\bP^1$
(the other three factors are points), the fiber of the fibration  $Y_2\to Y_3$ is $\bP^1\times \bP^1$
(with two trivial factors) and the fiber of the fibration  $Y_1\to Y_2$ is $\bP^1$ (with three trivial factors).  

Now let $J=\{1,3\}$. Then the dimension vector $d_J^{(i,r)}$ is given by

\medskip

\[
\begin{tikzcd}
1 & & 1 &  &  0 \\
 & 1 & 2  & 2 &  \\
1  & 2 & & 2 & 1 \\
 & 2  & 2 & 1 &\\
 0 & & 1 & & 1
\end{tikzcd}
\]

\medskip

The tower \eqref{tower} is as follows: $Y_4$ is a point, the fiber of the fibration $Y_3\to Y_4$ is $\bP^1\times \bP^1$ (with two trivial factors), the fiber of the fibration  $Y_2\to Y_3$ is a point
(the product of four points) and the fiber of the fibration  $Y_1\to Y_2$ is $\bP^1\times \bP^1$ (with two trivial factors).

\begin{rem}
Similar to Remark \ref{rem:non-smoothness-irreducible components} one can show that both $X_{\{1,2\}}(2,4)$ and $X_{\{1,3\}}(2,4)$ are singular.   

\end{rem}

\subsection{Geometric and Combinatorial properties of the Desingularization}\label{subsec:geo-prop-desing}

\begin{lem}\label{lem:aut-group-desing}
The automorphism group of $ \widetilde U_{[n]}$ satisfies
\[  {\rm Aut}_{\widetilde \Delta_n}\big( \widetilde U_{[n]}\big) \cong {\rm Aut}_{\Delta_n}\big(U_{[n]}\big). \]
\end{lem}

\begin{proof}
Composing $\beta_{i+1,2} \circ \alpha_{i,1}$ for all $i \in \mathbb{Z}_n$ we obtain the same relations on each component $A^{(i,1)}$ of $A \in {\rm Aut}_{\widetilde \Delta_n}( \widetilde U_{[n]})$ as for the component $B^{(i)}$ of $B \in {\rm Aut}_{\Delta_n}(U_{[n]})$ (see Proposition \ref{prop:endomorphism-algebra}). By construction of $\widetilde U_{[n]}$ all other components $A^{(i,r)}$ are the lower diagonal blocks of size $n-r+1$ in the matrices $A^{(i,1)}$. This gives us the desired isomorphism. 
\end{proof} 

Observe that the ${\rm G}_\mathbf{n}$ action on ${\rm R}_\mathbf{n}(\Delta_n)$ preserves the relations satisfied by the representations and the automorphism group of $M \in {\rm R}_\mathbf{n}(\Delta_n)$ is its ${\rm G}_\mathbf{n}$ stabilizer. Hence on both sides we obtain the same groups if we view $U_{[n]}$ and $\widetilde U_{[n]}$ as bounded quiver representations.

\begin{lem}
The strata in the quiver Grassmannian ${\rm Gr}_{\mathbf{d}_J} (\widetilde U_{[n]})$ are exactly the ${\rm Aut}_{\widetilde \Delta_n}( \widetilde U_{[n]})$ orbits. 
\end{lem} 

\begin{proof}
The representation $\widetilde U_{[n]}$ is an injective bounded $\widetilde \Delta_n$ representation for the relations as in Lemma~\ref{lem:bounding-relations-part-1} and  Lemma~\ref{lem:bounding-relations-part-2}. Hence we can apply \cite[Lemma~2.28]{Pue2019}. 
\end{proof} 

\begin{rem}
Analogous to Section~\ref{sec:Skeletal-T-Action} we define an action of $T=(\mathbb{C}^*)^{n+1}$ on the quiver Grassmannians ${\rm Gr}_{\mathbf{d}_J} (\widetilde U_{[n]})$, induced by the fact that the basis for the vector spaces of $\widetilde U_{[n]}$ are subsets in the basis for the vector spaces of $U_{[n]}$, and we can hence restrict the weight tuples. In particular, the desingularization map is $T$ equivariant with respect to this action. This is convenient if it comes to the computation of equivariant Euler classes (see \cite[Lemma~2.1.(3)]{LaPu2020}). 
\end{rem}

\begin{lem}
The ${\rm Aut}_{\widetilde \Delta_n}( \widetilde U_{[n]})$ orbits of the $T$ fixed points in the quiver Grassmannian ${\rm Gr}_{\mathbf{d}_J} (\widetilde U_{[n]})$ provide a cellular decomposition. 
\end{lem}

\begin{proof}
The vector space over the vertex $(i,r)$ of $\widetilde{\Delta}_n$ of any subrepresentation in ${\rm Gr}_{\mathbf{d}} (\widetilde U_{[n]})$ is an element of the Grassmannian of subspaces ${\rm Gr}_{d_J^{(i,r)}} ( \mathbb{C}^{n-r+1})$.
Let $p \in {\rm Gr}_{\mathbf{d}_J} (\widetilde U_{[n]})$ be a $T$ fixed point. By $p^{(i,r)} \in {\rm Gr}_{d_J^{(i,r)}} ( \mathbb{C}^{n-r+1})$ we denote its component over the vertex $(i,r)$ of $\widetilde{\Delta}_n$. Every element $A \in {\rm Aut}_{\widetilde \Delta_n}( \widetilde U_{[n]})$ acts on $p^{(i,r)}$ via the component $A^{(i,r)}$, which is a lower triangular matrix by Proposition~\ref{prop:endomorphism-algebra} and Lemma~\ref{lem:aut-group-desing}. Hence the orbit of $p^{(i,r)}$ in the Grassmannian of subspaces ${\rm Gr}_{d_J^{(i,r)}} ( \mathbb{C}^{n-r+1})$ is a cell. 

The orbit ${\rm Aut}_{\widetilde \Delta_n}( \widetilde U_{[n]}).p$ is the intersection of these cells along the maps of $\widetilde U_{[n]}$. Since the maps along arrows $\alpha_{i,r}$ are inclusions and the maps along $\beta_{i,r}$ are projections where the last coordinate is sent to zero, the intersection of the cells in the Grassmannians of subspaces is a cell of ${\rm Gr}_{\mathbf{d}_J} (\widetilde U_{[n]})$.

It remains to show that every ${\rm Aut}_{\widetilde \Delta_n}( \widetilde U_{[n]})$ orbit contains a $T$ fixed point. Locally every orbit in ${\rm Gr}_{d_J^{(i,r)}} ( \mathbb{C}^{n-r+1})$ contains a $T$ fixed point. From the explicit shape of the maps of $\widetilde U_{[n]}$ it follows that the intersection of these cells also contains a $T$ fixed point. 
Hence the stratification of  ${\rm Gr}_{\mathbf{d}_J} (\widetilde U_{[n]})$ into the isomorphism classes of subrepresentations is also a cellular decomposition indexed by the $T$ fixed points.
\end{proof}

\begin{rem}Let us  consider the $\bC^*$ action on ${\rm Gr}_{\mathbf{d}_J} (\widetilde U_{[n]})$ induced by the $\bC^*$ action \eqref{rem:C*-action} on $X(k,n)$.
The cells in the previous lemma coincide with the attracting sets of the fixed points as introduced in Section~\ref{sec:equi-cycle} for the cycle.
\end{rem}

\appendix
\section{Linear Degenerations}\label{sec:linear-degenenerations}
Analogous to \cite[Section~2]{CFFFR2017} we construct linear degenerations of the Grassmannian ${\rm  Gr}_k(n)$ of $k$-dimensional subspaces of $\mathbb{C}^n$: Recall that ${\rm  Gr}_k(n)$ is isomorphic to the quiver Grassmannian for $\Delta_n$ with the representation 
\[ 
M_{\rm id} 
:= \Big( \big( V^{(i)}:= \mathbb{C}^n \big)_{i \in \mathbb{Z}_n}, \big( V_{\alpha_i}:= {\rm id}_{\mathbb{C}^n} \big)_{i \in \mathbb{Z}_n} \Big) \in {\rm R}_\mathbf{n}(\Delta_n)\]
and dimension vector $\mathbf{k} = (k,\dots,k) \in \mathbb{Z}^n$. 
 This motivates our  definition of linear degenerations. Since there are multiple ways of realizing ${\rm  Gr}_k(n)$ as a quiver Grassmannian, this is of course only one possible definition.

\begin{dfn}A linear degeneration of ${\rm  Gr}_k(n)= {\rm Gr}_\mathbf{k}(M_{\rm id})$ is a quiver Grassmannian ${\rm Gr}_\mathbf{k}(M)$ for $M \in  {\rm R}_\mathbf{n}(\Delta_n)$. 
\end{dfn}
In particular,  $X(k,n) = {\rm Gr}_\mathbf{k}(U_{[n]})$ is a linear degeneration of ${\rm  Gr}_k(n)$. 

The group $G_\mathbf{n}$ acts on ${\rm R}_\mathbf{n}(\Delta_n)$ such that the isomorphism classes of the $\Delta_n$ representations with dimension vector $\mathbf{n}= (n,\dots,n) \in \mathbb{Z}^n$ are exactly the $G_\mathbf{n}$ orbits in ${\rm R}_\mathbf{n}(\Delta_n)$. Moreover two quiver Grassmannians are isomorphic if the corresponding quiver representations are isomorphic. Hence the $G_\mathbf{n}$ orbits in ${\rm R}_\mathbf{n}(\Delta_n)$ parametrize the isomorphism classes of linear degenerations of ${\rm  Gr}_k(n)$.

The rank tuple of $M \in {\rm R}_\mathbf{n}(\Delta_n)$ is 
\[ \mathbf{r} := \mathbf{r}(M) := \big( r_{i,\ell} := {\rm rank} \, M_{i+\ell-1} \circ \dots \circ M_i \big)_{(i,\ell) \in \mathbb{Z}_n \times \{0,\dots, n^2+1\}}. \]
It is sufficient to consider $\ell \leq n^2+1$ because the maximal nilpotent 
representations in ${\rm R}_\mathbf{n}(\Delta_n)$ are $U(i;n^2)$, i.e. the representations whose coefficient quiver is an equioriented string on $n^2$ points winding around the cycle and ending over the $i$-th vertex. This implies that the isomorphism classes of linear degenerations are parametrized by the rank tuples of $G_\mathbf{n}$ orbits in ${\rm R}_\mathbf{n}(\Delta_n)$, since isomorphic representations have the same rank tuple and two points with the same rank tuple are conjugated by \cite[p. 32]{Kempken1982}.

We define a partial order of the rank tuples by comparing their entries component wise. We denote the rank tuples of $M_{\rm id}$ and $U_{[n]}$ by $\mathbf{r}_{\rm id}$ and $\mathbf{r}_{[n]}$. The entries of $\mathbf{r}_{\rm id}$ are all equal to $n$ and the $(i,\ell)$-th entry of $\mathbf{r}_{[n]}$ equals $\max\{n-\ell, 0\}$. In the setting of linear degenerations of the flag variety of type A, the degenerations with rank tuple between the flag and the so called Feigin degeneration $\mathcal{F}l^a_{n+1}$ are all of the same dimension \cite[Theorem~A]{CFFFR2017}. 

The construction of the quiver representation $U_{[n]}$ is somehow analogous to the construction of $\mathcal{F}l_{n+1}^a$ \cite[Definition~2.5]{CFR12}. But despite that there are linear degenerations of ${\rm  Gr}_k(n)$ with rank tuple between $\mathbf{r}_{\rm id}$ and $\mathbf{r}_{[n]}$ which are not of dimension $k(n-k)$, as explained in the following example.

\begin{example}\label{ex:degenerations}
We collect some intermediate degenerations which are not of dimension $k(n-k)$. For $M = U(i;n^2)$ the quiver Grassmannian ${\rm Gr}_\mathbf{k}(M)$ consists of the single point $U(i;kn)$ and its rank tuple is between $\mathbf{r}_{\rm id}$ and $\mathbf{r}_{[n]}$.
For $r \in \mathbb{Z}_{\geq 0}$ with $r \leq n$ we define 
\[
M_r := \bigoplus_{i =1}^{n-r} U(i;n) \oplus V_{\Delta_n} \otimes \mathbb{C}^r
\] 
 and compute 
\[ \dim_\mathbb{C} {\rm Gr}_\mathbf{k}(M_r) = \max_{\ell \in \{0,1,\dots,r\}} \big\{ k(n-k) + (k-\ell)(\ell -r) - \ell(n-k+\ell -r)  \big\} \]
using \cite[Proposition~3.4]{Pue2020}. This only matches $k(n-k)$ if $r=n$ or $r=0$, which corresponds exactly to $M_{\rm id}$ and $U_{[n]}$, respectively.
\end{example}

We close by showing that the degeneration of the Grassmannian into $X(k,n)$ is not flat. 
Let us consider the case $n=3$, $k=1$. The classical Grassmannian (the projective plane in this
case) is embedded diagonally into $\bP^2\times\bP^2\times \bP^2$. In particular, the coordinate
ring is triply graded. The component of degree $(1,1,1)$ is of dimension $10=\dim_\bC S^3(\bC^3)$.
We can make it explicit by introducing the coordinates $x_i,y_i,z_i$, $i=1,2,3$ corresponding
to three projective planes. Then a basis of the $(1,1,1)$ component is formed by monomials
\begin{gather*}
x_1y_1z_1,\ x_1y_1z_2,\ x_1y_1z_3,\ x_1y_2z_2,\ x_1y_2z_3,\\ 
x_1y_3z_3,\ x_2y_2z_2,\ x_2y_2z_3,\ x_2y_3z_3,\ x_3y_3z_3.
\end{gather*}
Now let us look at the quiver Grassmannian $X(1,3)\subset \bP^2\times\bP^2\times \bP^2$ formed by 
triples of lines $(l_1,l_2,l_3)$ such that ${\rm pr}_1 l_1\subset l_2$, ${\rm pr}_2 l_2\subset l_3$,
${\rm pr}_3 l_3\subset l_1$. Then the following $11$ elements form a basis of the degree $(1,1,1)$ component:
\begin{gather*}
x_1y_1z_1,\ x_1y_2z_1,\ x_1y_2z_3,\ x_2y_2z_2,\ x_2y_2z_3,\\ 
x_2y_2z_1,\ x_2y_3z_3,\ x_2y_3z_1,\ x_1y_3z_3,\ x_1y_3z_1,\ x_3y_3z_3.
\end{gather*}
In particular, the two weight zero elements $x_1y_2z_3$ and $x_2y_3z_1$ are linearly independent. In fact, 
if one takes $l_1={\rm span}(a,b,c)$ with $(a,b,c)\ne 0$, then $l_2={\rm span}(0,b,c)$, $l_3={\rm span}(0,0,c)$
and the value of $x_1y_2z_3$ is equal to $abc$. However, the value of $x_2y_3z_1$ is zero.

\section{The case $k=1$}\label{sec:special-case}
In this appendix we consider the case  $k=1$ and arbitrary $n$ in more details.

For $k=1$, Grassmann necklaces are collections $(i_1,\dots,i_n)$, with $i_a\in[n]$ such that if $i_a\ne a$, then $i_{a+1}=i_a$ (as usual,
$i_{n+1}=i_1$). Recall that, by Theorem \ref{trm:irred-comp-and-dim}, the irreducible components $X_J(1,n)$ of the quiver Grassmannian $X(1,n)$ are labeled by $J=\{j\}$, $j\in [n]$. We use the notation $X_j(1,n):=X_{\{j\}}(1,n)$. 

\begin{prop}
One has
\begin{enumerate}
\item The total number of cells of $X(1,n)$ is equal to $2^n-1$.
\item The Poincar\'e polynomial is given by $(1+q)^n-q^n$.
\item All irreducible components of $X(1,n)$ are isomorphic. Each irreducible component 
$X_j(1,n)$ is isomorphic to a height $n$ Bott tower of $\bP^1$
fibrations over  a point.
\item \label{desiso} The desingularization map is an isomorphism over each irreducible component $X_j(1,n)$. 
\end{enumerate}
\end{prop}
\begin{proof}
The first two claims can be easily deduced from Theorem \ref{PP} and \cite{W05}. 
Let us prove the third claim: There are $n$ irreducible components and the rotation group
action on $X(k,n)$ induces the transitive action on the components $X_j(1,n)$. Hence 
all the irreducible components are isomorphic. Now let us consider the component $X_n(1,n)$.
Recall the second basis  of $U_{[n]}$ from Definition \ref{dfn:different-relisation-of-U_[n]}.
Let $v^{(a)}_1,\dots,v^{(a)}_n$ be the corresponding basis of $U_{[n]}^{(a)}$. Recall that we have identified 
all the spaces $U_{[n]}^{(a)}$ with $\bC^n$ in such a way that the map $U_{[n]}^{(a)}\to U_{[n]}^{(a+1)}$
sends $v^{(a)}_i$ to $v^{(a+1)}_{i+1}$ for $i<n$ and $v_n$ is sent to zero. We denote this map by $s_1$.
We also use the notation $s_{-1}$ for the map
sending $v_1$ to zero and $v^{(a)}_i$ to $v^{(a-1)}_{i-1}$ for $i>1$.
Then $X_n(1,n)$ consists of collections $(V_a)_{a=1}^n$ of lines in $\bC^n$ such that
\begin{itemize}
    \item $V_a\subset V_{a+1} + \bC v_n$,
    \item $V_a\subset \mathrm{span}(v_a,\dots,v_n)$.
\end{itemize}
Then one has $V_n=\bC v_n$, $V_{n-1}$ is an arbitrary line in the span of $v_n$ and $v_{n-1}$, 
$V_{n-2}$ is an arbitrary line in the span of $v_n$ and $s_{-1}V_{n-1}$, and in general
$V_a$ is an arbitrary line in the span of $v_n$ and $s_{-1}V_{a+1}$. This construction identifies
$X_n(1,k)$ with a Bott tower.

Finally, let us prove claim \eqref{desiso} for $j=n$. The dimensions $d_n^{(i,r)}:=d_{\{n\}}^{(i,r)}$ are easily computed
as 
\[
d_n^{(i,r)}=\begin{cases}1, & r\le i\\ 0, & r>i.\end{cases}
\]
The quiver Grassmannian ${\rm Gr}_{{\bf d}_n}(\widetilde U_{[n]})$ (which desingularizes the component 
$X_n(1,n)$) consists of collections of subspaces $(V^{(i,r)})$ such that (in particular)
$V^{(i,r)}\subset V^{(i,r-1)}$ and the non-trivial (i.e. one-dimensional) subspaces correspond to $r\le i$. 
Therefore, for $r\le i$ one has $V^{(i,r)}=V^{(i,1)}$, which means that the map 
${\rm Gr}_{{\bf d}_n}(\widetilde U_{[n]})\to X_n(1,n)$ is one-to-one.
\end{proof}
In the $k=1$ case, also the cell dimensions have a particularly nice behavior. For an element $i_\bullet=(i_1, \ldots, i_n)\in\mathcal{GN}_{1,n}$, we denote by $n_{i_\bullet}$ the number of distinct entries of $i_\bullet$. 
\begin{example}
If $i_\bullet=(1,1, \ldots, 1)\in \mathcal{GN}_{1,n}$, then $n_{i_\bullet}=1$. On the other hand, if $i_\bullet=(1,2,3,\ldots,n)\in \mathcal{GN}_{1,n}$, then $n_{i_\bullet}=n$.
\end{example}

The following lemma might be known to the experts, but we were not able to find it in the literature and therefore we decided to include it.
 
\begin{lem}
Let $i_\bullet\in\mathcal{GN}_{1,n}$. Then,
\[
\dim_\bC C(i_\bullet)=n-n_{i_\bullet}.
\]
\end{lem}
\begin{proof}
We proceed by induction on $n-n_{i_\bullet}$.  If $n_{i_\bullet}=n$, then necessarily $i_\bullet=(1,2,\ldots,n)$, which corresponds to the (unique) closed cell, having dimension $0=n-n_{i_\bullet}$, as desired. 

Suppose now that $n-n_{i_\bullet}>0$.
If $n-n_{i_\bullet}=1$, then $n_{i_\bullet}=1$ and there exists a $j\in[n]$ such that $\overline{C(i_\bullet)}=X_j(1,n)$, and hence $\dim_\bC(C(i_\bullet))=\dim_\bC(X(1,n))=n-1$ as claimed.

Thus, we assume that $1<n_{i_\bullet}<n$. In this case, there exists a $t\in[n]$ such that $i_j\neq t$ for any $j\in[n]$. In particular, $i_t\neq t$ and, by definition of Grassmann necklace, it must hold $i_t=i_{t+1}$.   

Denote by $h:=\max_t\{a\mid i_a=a\}$ (this is well-posed as $n_{i_\bullet}\neq 1$) and consider $i_\bullet'\in\mathcal{GN}_{1,n}$ given by
\[
i'_r=\left\{
\begin{array}{ll}
i_r    & \hbox{ if }r\in[h]_{t+1},\\
t     & \hbox{ if } r\not\in[h]_{t+1}.
\end{array}
\right.
\]
By construction, $n_{i'_{\bullet}}=n_{i_\bullet}+1$ and so we can assume by induction that $\dim_\bC C(i'_\bullet)=n-n_{i'_\bullet}$. 

For $j_\bullet\in \mathcal{GN}_{1,n}$, we denote by $f_{j_\bullet}$ the corresponding bounded permutation as in \S\ref{subsec:BoundedPerms}. Observe that 
\[
f_{i_\bullet}(a)=f_{i'_\bullet}(a), \quad\Longleftrightarrow a\not\in \{h+n\bZ\}\cup \{t+n\bZ\},
\]
where $f_{i_\bullet}(a)=a$ for all $a\in [t-1]_{h+1}+n\bZ$.
Moreover, if we set $l:=i_t$, then
\[
f_{i_\bullet}(t)=t, \quad f_{i'_\bullet}(t)=\left\{
\begin{array}{ll}
l   & \hbox{ if }t<l,\\
l+n     & \hbox{otherwise}
\end{array}
\right.
\]
and
\[
f_{i_\bullet}(h)=\left\{
\begin{array}{ll}
l   & \hbox{ if }h<l,\\
l+n     & \hbox{otherwise}
\end{array}
\right., \quad f_{i'_\bullet}(h)=\left\{
\begin{array}{ll}
k   & \hbox{ if }h<t,\\
k+n     & \hbox{otherwise}
\end{array}
\right.
\]
Finally, we notice that
\begin{align*}
\ell(f_{i_\bullet})&=|\{(i,j)\in[n]\times\bZ:\ i<j \text{ and } f_{i_\bullet}(i)>f_{i_\bullet}(j)\}|\\
&=|\{(i,j)\in[n]\times\bZ:\ i<j \text{ and } f_{i'_\bullet}(i)>f_{i_\bullet}(j)\}\cup\{(h,m)\}|\\
&=\ell(f_{i'_\bullet})+1, 
\end{align*}
where $m=t$ if $h<t$ and $m=t+n$ otherwise. The claim now follows by the inductive hypothesis together with Theorem \ref{trm:cell-dim-quiver-Grass}.
\end{proof}


\end{document}